\documentclass[11pt]{amsart}
\usepackage[babel]{csquotes}
\usepackage{enumitem}
\usepackage{amsmath,amsthm,amssymb,mathrsfs,amsfonts,verbatim,enumitem,color,leftidx,mathtools}
\usepackage{mathabx}
\usepackage{etoolbox} 
\usepackage{bbm}
\usepackage[all,tips]{xy}
\usepackage{graphicx,ifpdf}
\usepackage{stmaryrd}
\usepackage{amssymb}

\ifpdf
\fi
\usepackage[colorlinks]{hyperref}
\hypersetup{
linkcolor=blue,          % color of internal links (change box color with linkbordercolor)
citecolor=green,        % color of links to bibliography
}

\usepackage{tikz}%pictures drawn in latex

\newtheorem{thm}{Theorem}[section]
\newtheorem{lem}[thm]{Lemma}
\newtheorem{cor}[thm]{Corollary}
\newtheorem{prop}[thm]{Proposition}

\newtheorem*{thm*}{Dirichlet's Theorem}
\newtheorem*{cor*}{Dirichlet's Corollary}
\newtheorem*{thm1*}{Khintchine-Groshev Theorem}

\theoremstyle{definition}
\newtheorem{defi}[thm]{Definition}

\theoremstyle{remark}
\newtheorem{rem}[thm]{Remark}

\numberwithin{equation}{section}
\definecolor{esperance}{rgb}{0.0,0.5,0.0}

\newcommand{\bp}{\mathbf{p}}
\newcommand{\bq}{\mathbf{q}}

\newcommand{\bs}{\mathbf{s}}

\newcommand{\bx}{\mathbf{x}}
\newcommand{\by}{\mathbf{y}}

\newcommand{\R}{\mathbb{R}}

\newcommand{\Z}{\mathbb{Z}}

\newcommand{\de}{\delta}

\newcommand{\al}{\alpha}
\newcommand{\ga}{\gamma}

\newcommand{\del}{\delta}
\newcommand{\Del}{\Delta}
\newcommand{\lam}{\lambda}
\newcommand{\Lam}{\Lambda}
\newcommand{\eps}{\epsilon}

\newcommand{\ka}{\kappa}

\newcommand{\om}{\omega}

\newcommand{\vphi}{\varphi}
\newcommand{\cA}{\mathcal{A}}

\newcommand{\cC}{\mathcal{C}}

\newcommand{\cH}{\mathcal{H}}

\newcommand{\cR}{\mathcal{R}}
\newcommand{\cS}{\mathcal{S}}

\newcommand{\bR}{\mathbb{R}}
\newcommand{\bZ}{\mathbb{Z}}
\newcommand{\bQ}{\mathbb{Q}}

\newcommand{\bN}{\mathbb{N}}

\newcommand\wh[1]{\widehat{#1}}
\newcommand\set[1]{\left\{#1\right\}}

\newcommand\mb[1]{\mathbf{#1}}

\newcommand\tb[1]{\textbf{#1}}

\newcommand{\onto}{\xymatrix{\ar@{>>}[r]&}}
\newcommand{\eq}[1]
{
\begin{equation*}
{#1}
\end{equation*}
}
\newcommand{\eqlabel}[2]
{
\begin{equation}
{#2}\label{#1}
\end{equation}
}
\makeatletter
\newcommand*{\rom}[1]{\expandafter\@slowromancap\romannumeral #1@}
\makeatother

\begin{document}

\title[Hausdorff measures and Dirichlet non-improvable affine forms]{Hausdorff measure of sets of Dirichlet non-improvable affine forms}

\author{Taehyeong Kim}
\address{Taehyeong ~Kim. Department of Mathematical Sciences, Seoul National University, {\it kimth@snu.ac.kr}}
\author{Wooyeon Kim}
\address{Wooyeon Kim. Department of Mathematics, ETH Z\"{u}rich, 
{\it wooyeon.kim@math.ethz.ch}}

% \thanks will become a 1st page footnote.
\thanks{}

%\author{}


\def\thefootnote{}
   %%d 
\def\thefootnote{\arabic{footnote}}
\setcounter{footnote}{0}

\begin{abstract}
For a decreasing real valued function $\psi$, a pair $(A,\mb{b})$ of a real $m\times n$ matrix $A$ and $\mb{b}\in\bR^m$ is said to be $\psi$-Dirichlet improvable if the system
$$\|A\mb{q}+\mb{b}-\mb{p}\|^m < \psi(T)\quad\text{and}\quad\|\mb{q}\|^n < T$$
has a solution $\mb{p}\in\bZ^m$, $\mb{q}\in\bZ^n$ for all sufficiently large $T$, where $\|\cdot\|$ denotes the supremum norm. Kleinbock and Wadleigh (2019) established an integrability criterion for the Lebesgue measure of the $\psi$-Dirichlet non-improvable set. In this paper, we prove a similar criterion for the Hausdorff measure of the $\psi$-Dirichlet non-improvable set. Also, we extend this result to the singly metric case that $\mb{b}$ is fixed. As an application, we compute the Hausdorff dimension of the set of pairs $(A,\mb{b})$ with uniform Diophantine exponents $\wh{w}(A,\mb{b})\leq w$.
\end{abstract}

\maketitle
\section{Introduction}

Let $m,n$ be positive integers, and denote by $M_{m,n}(\bR)$ the set of $m\times n$ real matrices. Dirichlet's Theorem (1842) is the starting point of this paper:

\begin{thm*}
For any $A\in M_{m,n}(\bR)$ and $T>1$, there exist $\mb{p}\in\bZ^m$ and $\mb{q}\in\bZ^n \setminus \{\mb{0}\}$ such that 
\eqlabel{DT}{
\|A\mb{q}-\mb{p}\|^{m}\leq \frac{1}{T}\quad\text{and}\quad \|\mb{q}\|^{n}< T.
}
\end{thm*}
Here and hereafter, $\|\cdot\|$ stands for the supremum norm on $\bR^k$ given by $\|\mb{x}\|=\max_{1\leq i\leq k}|x_i|$. Dirichlet's theorem is a \textit{uniform} Diophantine approximation result because it guarantees a non-trivial integral solution for \textit{all} $T$. A weaker form guaranteeing that such a system is solvable for an \textit{unbounded} set of $T$ is called \textit{asymptotic} approximation. The following corollary is an example of asymptotic approximation from Dirichlet's theorem. 

\begin{cor*}
For any $A\in M_{m,n}(\bR)$ there exist infinitely many $\mb{q}\in\bZ^n$ such that 
\eqlabel{DC}{
\|A\mb{q}-\mb{p}\|^{m} < \frac{1}{\|\mb{q}\|^n}\quad\text{for some }\mb{p}\in\bZ^m.
}
\end{cor*}

The above two statements give a rate of approximation which works for all real matrices. However, if we replace the right-hand sides of \eqref{DT} and \eqref{DC} by faster decaying functions of $T$ and $\|\mb{q}\|^n$ respectively, then one can ask sizes of corresponding sets of matrices satisfying the improved systems. Those sets are very well studied in the setting of Dirichlet's Corollary.

For a function $\psi:\bR_{+}\to\bR_{+}$, we say that $A\in M_{m,n}(\bR)$ is $\psi$-\textit{approximable} if there exist infinitely many $\mb{q}\in \bZ^n$ such that \footnote{Here, we follow the definition given in \cite{KM99, KW19} but, in Section 4, we will use the slightly different definition, such as \cite{BV10}, where the inequality $\|A\mb{q}-\mb{p}\| < \psi(\|\mb{q}\|)$ is used instead of \eqref{approx}.\label{fn1}}
\eqlabel{approx}{
\|A\mb{q}-\mb{p}\|^m < \psi(\|\mb{q}\|^n)\quad\text{for some }\mb{p}\in\bZ^m.
}
Denote by $W_{m,n}(\psi)$ the set of $\psi$-approximable matrices in the unit cube $[0,1]^{mn}$. Then the set $W_{m,n}(\psi)$ satisfies the following zero-one law with respect to the Lebesgue measure.

\begin{thm1*}\label{KG}
Given a non-increasing $\psi$, the set $W_{m,n}(\psi)$ has zero (resp. full) Lebesgue measure if and only if the series $\sum_{k}\psi(k)$ converges (resp. diverges).
\end{thm1*}

To distinguish between sizes of null sets, we can consider Hausdorff measure and dimension as the appropriate tools. Since the set $W_{m,n}(\psi)$ is always containing $m(n-1)$-dimensional hyperplanes, we may focus on $s$-dimensional Hausdorff measures with $s>m(n-1)$. The following result was proved by Jarn\'{i}k in 1931 for $n=1$ and \cite{DV97} in general.

\begin{thm}[Jarn\'{i}k]\label{Jarnik} 
Let $\psi$ be a non-increasing function. Then for $s>m(n-1)$,
\eq{
\cH^{s}(W_{m,n}(\psi))=
  \begin{dcases}
    0       & \quad \text{if } \sum_{q=1}^{\infty} q^{m+n-1}\left(\frac{\widehat{\psi}(q)}{q}\right)^{s-m(n-1)}<\infty,\\
    \cH^{s}([0,1]^{mn})  & \quad \text{if } \sum_{q=1}^{\infty} q^{m+n-1}\left(\frac{\widehat{\psi}(q)}{q}\right)^{s-m(n-1)}=\infty,
  \end{dcases} 
}  where $\widehat{\psi}(q)=\psi(q^n)^{\frac{1}{m}}$.
\end{thm}
Here, $\cH^{s}([0,1]^{mn})$ is infinity for $s<mn$. On the other hand, $\cH^{mn}$ comparable to the $mn$-dimensional Lebesgue measure, hence, Theorem \ref{Jarnik} implies Khintchine-Groshev Theorem.   	 

It is worth mentioning that Jarn\'{i}k's theorem was indeed proved for any dimension functions $f$, not just the functions of the form $f(r):=r^s$ stated in Theorem \ref{Jarnik}, see \cite{DV97}.  Regarding inhomogeneous Diophantine approximation, the analogue of Jarn\'{i}k's theorem for doubly metric case was proved in \cite{HKS20} and for singly metric case in \cite{B04}.

For similar generalizations in the setting of Dirichlet's Theorem, let us give the following definition: for a non-increasing function $\psi:[T_0,\infty)\to\bR_{+}$, where $T_0 >1$ is fixed, we say that $A\in M_{m,n}(\bR)$ is $\psi$-\textit{Dirichlet} if the system
\eq{
\|A\mb{q}-\mb{p}\|^m < \psi(T)\quad\text{and}\quad\|\mb{q}\|^n < T
}
has a nontrivial integral solution for all large enough $T$. Surprisingly, no zero-one law analogous to Khintchine-Groshev Theorem was known until recently when Kleinbock and Wadleigh \cite{KW18} proved a zero-one law on the Lebesgue measure of Dirichlet improvable numbers, that is, $m=n=1$. The Hausdorff measure-theoretic results for Dirichlet non-improvable numbers analogous to Theorem \ref{Jarnik} have also been established in \cite{HKWW18} for a general class of dimension functions $f$ called the essentially sub-linear dimension functions. For the non-essentially sub-linear dimension functions, the relevant results are in \cite{BHS20}. For general $m,n\in\bN$, Kleinbock, Str\"{o}mbergsson, and Yu \cite{KSY21} recently gave sufficient conditions on $\psi$ to ensure that the set of $\psi$-Dirichlet $m\times n$ matrices has zero or full Lebesgue measure.

Now, we focus our attention on inhomogeneous Diophantine approximation replacing the values of a system of linear forms $A\mb{q}$ by those of a system of \textit{affine forms} $\mb{q}\mapsto A\mb{q}+\mb{b}$, where $A\in M_{m,n}(\bR)$ and $\mb{b}\in\R^m$.
Let $\widetilde{M}_{m,n}(\bR):= M_{m,n}(\bR)\times \bR^m$. Following \cite{KW19}, for a non-increasing function $\psi:[T_0,\infty)\to\bR_{+}$, we say that a pair $(A,\mb{b})\in \widetilde{M}_{m,n}(\bR)$ is $\psi$-\textit{Dirichlet} if there exist $\mb{p}\in\bZ^m$ and $\mb{q}\in\bZ^n$ such that
\eqlabel{dirichlet}{
\|A\mb{q}+\mb{b}-\mb{p}\|^m < \psi(T)\quad\text{and}\quad\|\mb{q}\|^n < T
}
whenever $T$ is large enough. Denote by $\widehat{D}_{m,n}(\psi)$ the set of $\psi$-Dirichlet pairs in the unit cube $[0,1]^{mn+m}$. Note that in this definition, the case $\mb{q}=0$ is allowed so that $(A,\mb{b})$ is always $\psi$-Dirichlet for any $\mb{b}\in\bZ^m$.

Recently, Kleinbock and Wadleigh established the following zero-one law for the set $\widehat{D}_{m,n}(\psi)$ with respect to the Lebesgue measure.
\begin{thm}\cite{KW19}\label{KW}
Given a non-increasing $\psi$, the set $\widehat{D}_{m,n}(\psi)$ has zero (resp. full) Lebesgue measure if and only if the series
\eqlabel{KWsum}{
\sum_{j}\frac{1}{\psi(j)j^2}
}
diverges (resp. converges).
\end{thm}

As mentioned in \cite[Section 7]{KW19}, one can naturally ask whether Theorem \ref{KW} can be extended along two directions:
\begin{itemize}
    \item Zero-infinity law for a Hausdorff measure,
    \item Singly metric case ($\mb{b}$ fixed).
\end{itemize}
Although Theorem \ref{KW} provides the Lebesgue measure of the set $\widehat{D}_{m,n}(\psi)$, nothing was known about the Hausdorff dimension of this set. In this article, we give an analogue of Theorem \ref{KW} for the Hausdorff measure by establishing the zero-infinity law analogous to Theorem \ref{Jarnik}. Let us state our main theorem.

\begin{thm}\label{thm1}
Given a decreasing function $\psi$ with $\lim_{T\to\infty}\psi(T)=0$ and $0\leq s\leq mn+m$, the $s$-dimensional Hausdorff measure of $\widehat{D}_{m,n}(\psi)^c$ is given by
\eq{
\cH^{s}(\widehat{D}_{m,n}(\psi)^c)=
  \begin{dcases}
    0       & \quad \text{if } \sum_{q= 1}^{\infty}\frac{1}{\psi(q)q^2}\left(\frac{q^{\frac{1}{n}}}{\psi(q)^{\frac{1}{m}}}\right)^{mn+m-s}<\infty,\\
    \cH^{s}([0,1]^{mn+m})  & \quad \text{if } \sum_{q= 1}^{\infty}\frac{1}{\psi(q)q^2}\left(\frac{q^{\frac{1}{n}}}{\psi(q)^{\frac{1}{m}}}\right)^{mn+m-s}=\infty.
  \end{dcases} 
} Moreover, the convergent case still holds for every non-increasing function $\psi$ without the assumption $\lim_{T\to\infty}\psi(T)=0$.
\end{thm}

On the other hand, Theorem \ref{KW} provides us only the information on Lebesgue measure in the doubly metric case, i.e. it computes Lebesgue measure of the set $\widehat{D}_{m,n}(\psi)\subseteq \widetilde{M}_{m,n}$. A more refined question in inhomogeneous Diophantine approximation is fixing $\mb{b}\in\bR^m$ and asking the analogous question for the slices of $\widehat{D}_{m,n}(\psi)$. For fixed $\mb{b}\in \bR^m$, let $\widehat{D}_{m,n}^{\mb{b}}(\psi):=\{A\in M_{m,n}(\bR):(A,\mb{b})\in\widehat{D}_{m,n}(\psi)\}$. The following theorem answers the question for the singly metric case.

\begin{thm}\label{thm2}
Given a decreasing function $\psi$ with $\lim_{T\to\infty}\psi(T)=0$ and $0\leq s\leq mn$, the $s$-dimensional Hausdorff measure of $\widehat{D}_{m,n}(\psi)^c$ is given by
\eq{
\cH^{s}(\widehat{D}_{m,n}^{\mb{b}}(\psi)^c)=
  \begin{dcases}
    0       & \quad \text{if } \sum_{q= 1}^{\infty}\frac{1}{\psi(q)q^2}\left(\frac{q^{\frac{1}{n}}}{\psi(q)^{\frac{1}{m}}}\right)^{mn-s}<\infty,\\
    \cH^{s}([0,1]^{mn})  & \quad \text{if } \sum_{q= 1}^{\infty}\frac{1}{\psi(q)q^2}\left(\frac{q^{\frac{1}{n}}}{\psi(q)^{\frac{1}{m}}}\right)^{mn-s}=\infty.
  \end{dcases} 
}
for every $\mb{b}\in\bR^m\setminus\bZ^m$. Moreover, the convergent case still holds for every $\mb{b}\in\bR^m$ and every non-increasing function $\psi$ without the assumption $\lim_{T\to\infty}\psi(T)=0$.
\end{thm}

Theorem \ref{thm1} and Theorem \ref{thm2} can be applied to compute the Hausdorff dimension of the Dirichlet non-improvable set for some specific functions explicitly. For example, let $\psi_a(q):=q^{-a}$ and $\psi_{a,b}(q):=q^{-a}(\log q)^b$ for $a>0$, $b\ge 0$.
Our results directly gives the following:
For $0< a\leq1$, the Hausdorff dimension of $\widehat{D}_{m,n}^{\mb{b}}(\psi_{a,b})^c$ is $s_a:=mn-\frac{mn(1-a)}{m+na}$ and 
\eq{
\cH^{s_a}(\widehat{D}_{m,n}^{\mb{b}}(\psi_{a,b})^c)=
  \begin{dcases}
    0       & \quad \text{if } b>\frac{m+na}{m+n},\\
    \cH^{s_a}([0,1]^{mn})  & \quad \text{if } b\leq\frac{m+na}{m+n}
  \end{dcases} 
}
for every $\mb{b}\in\bR^m\setminus\bZ^m$. More explicitly, $\cH^{s_a}([0,1]^{mn})=\cH^{mn}([0,1]^{mn})\asymp 1$ if $a=1$, and $\cH^{s_a}([0,1]^{mn})=\infty$ otherwise. For the doubly metric case, the Hausdorff dimension of $\widehat{D}_{m,n}(\psi_{a,b})^c$ is $s_a+m$ and
\eq{
\cH^{s_a+m}(\widehat{D}_{m,n}(\psi_{a,b})^c)=
  \begin{dcases}
    0       & \quad \text{if } b>\frac{m+na}{m+n},\\
    \cH^{s_a+m}([0,1]^{mn+m})  & \quad \text{if } b\leq\frac{m+na}{m+n}.
  \end{dcases} 
}
Also, we can observe that the Hausdorff dimension is always bigger than $mn-n$ for the singly metric case and $mn+m-n$ for the doubly metric case regardless of the choice of $\psi$.

We remark that the above results for $\psi_a$ can be stated in terms of \textit{uniform Diophantine exponents}. We denote by $\wh{w}(A,\mb{b})$ the supremum of the real numbers $w$ for which, for all sufficiently large $T$, the inequalities
$$\|A\mb{q}+\mb{b}-\mb{p}\| < T^{-w} \quad\text{and}\quad\|\mb{q}\| < T$$
have an integral solution $\mb{p}\in\bZ^m$ and $\mb{q}\in\bZ^n$. For further details and references regarding the above notion, see \cite{BL05, B16}. Considering $\psi_{a}$ with $a=\frac{mw}{n}$, we have the following corollary by the definition. 

\begin{cor}\label{exponent}
For any $w>0$,
$$\dim_H\set{(A,\mb{b})\in M_{m,n}(\bR)\times\bR^m: \wh{w}(A,\mb{b})\leq w}=\min\left\{mn+m-\frac{n-mw}{1+w}, mn+m\right\},$$
$$\dim_H\set{A\in M_{m,n}(\bR): \wh{w}(A,\mb{b})\leq w}= \min\left\{mn-\frac{n-mw}{1+w},mn\right\}$$
 for every $\mb{b}\in\bR^m\setminus\bZ^m$. Therefore, for any $0<w\leq\frac{n}{m}$,
$$\dim_H\set{(A,\mb{b})\in M_{m,n}(\bR)\times\bR^m: \wh{w}(A,\mb{b})=w}=mn+m-\frac{n-mw}{1+w},$$
$$\dim_H\set{A\in M_{m,n}(\bR): \wh{w}(A,\mb{b})=w}= mn-\frac{n-mw}{1+w}$$
 for every $\mb{b}\in\bR^m\setminus\bZ^m$.

\end{cor}

The structure of this paper is as follows: In Section 2, we introduce some preliminaries including ubiquitous systems and reduce the inhomogeneous Diophantine approximation problem to a shrinking target problem in the space of grids in $\bR^{m+n}$. In Section 3 and Section 4, we prove the convergent part and the divergent part of the main theorem, respectively. In both of the proofs of the convergence and the divergence, our main tool is a duality phenomenon between the homogeneous Diophantine approximation and the inhomogeneous Diophantine approximation, which is called Transference Principle (see \cite{Cas57, BL05}). Roughly speaking, $^{t}A\in M_{m,n}(\bR)$ is well-approximable if and only if $(A,\mb{b})\in M_{m,n}(\bR)\times\bR^m$ is Dirichlet non-improvable for most of $\mb{b}\in\bR^m$. In Section 3, we count the number of covering balls following the method used in \cite{KKLM} based on this observation. In Section 4, we use mass distributions on some well-approximable sets for the doubly metric case and establish a local ubiquity with an appropriate ubiquitous function for the singly metric case. Finally, in Section 5, we discuss some concluding remarks and possible open questions.

\vspace{5mm}
\tb{Acknowledgments}. 
We would like to express our gratitude to Seonhee Lim for introducing us to the subject and for her constant help and encouragement. We also would like to thank Yann Bugeaud, Jiyoung Han, and Dmitry Kleinbock for providing helpful comments. We are deeply indebted to the referee for pointing out that the proof of Lemma 4.6 in the original manuscript was wrong. 

TK was supported by Samsung Science and Technology Foundation under Project No. SSTF-BA1601-03, and the National Research Foundation of Korea under Project Number NRF-2020R1A2C1A01011543, and NRF-2021R1A6A3A13039948. WK was supported by the Korea Foundation for Advanced Studies.

\section{Preliminaries}

\subsection{Hausdorff measure and auxiliary lemmas}
Below we give a brief introduction to Hausdorff measure and dimension. For further details, see \cite{Fal}.

Let $E$ be a subset of a Euclidean space $\bR^k$. For $\del>0$, a cover $\cC$ of $E$ is called a $\del$-\textit{cover} of $E$ if $\text{diam}(C)\leq\del$ for all $C\in\cC$. For $0\leq s\leq k$, let
\eq{
\cH_{\del}^{s}(E)=\inf\sum_{C\in\cC}\text{diam}(C)^s,
}
where the infimum is taken over all finite or countable $\del$-cover $\cC$ of $E$. Then the $s$-\textit{dimensional Hausdorff measure} of a set $E$ is defined by
\eq{
\cH^{s}(E)=\lim_{\del\to 0}\cH_{\del}^{s}(E).
}
Finally, the \textit{Hausdorff dimension} of $E$ is given by
\eq{
\text{dim}_{H}(E)=\inf\{s\geq0:\cH^{s}(E)=0\}.
}

The following principle commonly known as the Mass Distribuiton Principle \cite[\S 4.1]{Fal} will be used to show the divergent part of Theorem \ref{thm1}.
\begin{lem}[Mass Distribution Principle]\label{MDP}
Let $\mu$ be a probability measure supported on $F\subset\bR^k$. Suppose there are positive constants $c>0$, $r_0>0$, and $0\leq s\leq k$ such that
\eq{
\mu(B)\leq cr^s
}
for any ball $B$ with radius $r\leq r_0$. If $E$ is a subset of $F$ with $\mu(E)=\lambda > 0$ then $\cH^{s}(E)\geq\lambda/c$.
\end{lem}

We state the Hausdorff measure version of the Borel-Cantelli lemma \cite[Lemma 3.10]{BD99} which will allow us to estimate the Hausdorff measure
of the convergent part of Theorem \ref{thm1} and Theorem \ref{thm2}. 
\begin{lem}[Hausdorff-Cantelli]\label{hcan}
Let $\{B_i\}_{i\geq1}$ be a sequence of measurable sets in $\bR^k$ and suppose that for some $0\leq s\leq k$,
\eq{
\sum_{i}\textup{diam}(B_i)^s< {\infty}.
}
Then 
\eq{
\cH^{s}(\limsup_{i\to\infty}B_i)=0.
}
\end{lem}

\subsection{Homogeneous dynamics}
Our argument is based on the Dani correspondence, which forms a connection between Diophantine approximation and homogeneous dynamics. The classical Dani correspondence for homogeneous Diophantine approximation dates back to \cite{D85} (See also \cite{KM99}). The analogous correspondence between inhomogeneous Diophantine approximation and the dynamics in the space of grids have been used in \cite{Kl99, Sha11, ET11, LSS19, GV}. In this section, we introduce the space of grids in $\bR^{m+n}$ and the diagonal flow on this space. For $d=m+n$, let 
\eq{
G_{d}=SL_{d}(\bR)\quad\text{and}\quad\widehat{G}_{d}=ASL_{d}(\bR)=G_{d}\rtimes\bR^d,
}
and let
\eq{
\Gamma_{d}=SL_{d}(\bZ)\quad\text{and}\quad\widehat{\Gamma}_{d}=ASL_{d}(\bZ)=\Gamma_{d}\rtimes\bZ^d.
}
Elements of $\widehat{G}_{d}$ will be denoted by $\langle g,\mb{w} \rangle$, where $g\in G_d$ and $\mb{w}\in\bR^d$.
Denote by $X_d=G_{d} / \Gamma_{d}$ the space of unimodular lattices in $\bR^d$ and denote by $\widehat{X}_d=\widehat{G}_{d} / \widehat{\Gamma}_{d}$ the space of unimodular grids, i.e. affine shifts of unimodular lattices in $\bR^d$. For simplicity, let $G:=G_d$, $X:=X_d$ and denote by $m_X$ the Haar probability measure on $X_d$.
For $t\in\bR$, let
\eq{
a_t:=\text{diag}(e^{t/m},\dots,e^{t/m},e^{-t/n},\dots,e^{-t/n}).
}
Let us denote by 
\eq{u_{A}:= \left(\begin{matrix}
I_m & A \\ 
0 & I_n \\
\end{matrix}\right)\in G_d
\quad\text{and}\quad 
u_{A,\mb{b}}:=\left\langle
\left(\begin{matrix}
I_m & A \\ 
0 & I_n \\
\end{matrix}\right),
\left(\begin{matrix}
\mb{b} \\ 
0 \\
\end{matrix}\right)
\right\rangle\in\widehat{G}_d
}
for $A\in M_{m,n}(\bR)$ and $(A,\mb{b})\in\widetilde{M}_{m,n}(\bR)$.
Let us also denote by
\eq{
\Lam_A := u_A\Z^d\in X \quad\text{and}\quad\Lam_{A,\mb{b}} := u_{A,\mb{b}}\Z^d \in \widehat{X}_d,
}
where $u_{A,\mb{b}}\Z^d = \left\{\left(\begin{matrix}
A\mb{q}+\mb{b}-\mb{p}\\
\mb{q}\\
\end{matrix}\right): \mb{p}\in\bZ^m, \mb{q}\in\bZ^n
\right\}$.
The \textit{expanding horospherical subgroup} of $G_d$ with respect to $\set{a_t:t>0}$ is given by $H:=\set{u_A: A\in M_{m,n}(\bR)}$. On the other hand, the \textit{nonexpanding horospherical subgroup} of $G_d$ with respect to $\set{a_t:t>0}$ is given by $$\tilde{H}:=\set{\left(\begin{matrix}
P & 0\\
R & Q\\
\end{matrix}\right): P\in M_{m,m}(\bR), Q\in M_{n,n}(\bR), R\in M_{n,m}(\bR), \det(P)\det(Q)=1}.$$
Note that $\tilde{H}$ is the complementary subgroup to $H$. We denote by $m_H$ and $m_{\tilde{H}}$ the left-invariant Haar measure on $H$ and $\tilde{H}$, respectively.

Let $\mb{d}$ be a right invariant metric on $G$. We can take $\mb{d}$ to satisfy $||g-id||\leq\mb{d}(g,id)$ for $g$ in the sufficiently small ball $B_r^{G}(id)$, where $||\cdot||$ is the supremum norm on $M_{d,d}(\bR)$. This metric induces metrics on $H,\tilde{H}$, and $X$ by restriction. We let $B_r^{G}(id),B_r^{H}(id),B_r^{\tilde{H}}(id)$, and $B_r^{X}(id)$ denote the open ball in $G,H,\tilde{H}$, and $X$ of radius $r$ centered at the identity, respectively.

Following \cite{KW19}, we define the functions $\Del:\widehat{X}_d \to [-\infty,\infty)$ by
\eq{
\Del(\Lam):=\log\inf_{\mb{v}\in\Lam}\|\mb{v}\|,
}
which can be considered as the logarithm of a height function.

\begin{lem}\label{KM}\cite[Lemma 8.3]{KM99}
Let $m,n\in\bN$ and $T_0\in\R_{+}$ be given. Suppose $\psi:[T_0,\infty)\to\R_{+}$ is a continuous, non-increasing function. Then there exists a unique continuous function
\eq{
z=z_{\psi}:[t_0,\infty)\to\R,
}
where $t_0:=\frac{m}{m+n}\log{T_0}-\frac{n}{m+n}\log{\psi(T_0)}$, such that
\begin{enumerate}
	\item\label{p1} the function $t\mapsto t+nz(t)$ is strictly increasing and unbounded;
	\item\label{p2} the function $t\mapsto t-mz(t)$ is nondecreasing;
	\item\label{p3} $\psi(e^{t+nz(t)})=e^{-t+mz(t)}$ for all $t\geq t_0$.
\end{enumerate}
\end{lem}

The following lemma reduces the inhomogeneous Diophantine approximation problem to the shrinking target problem on the space of grids.
\begin{lem}\label{Dani}\cite{KM99, KW19}
Let $\psi:[T_0,\infty)\to\R_{+}$ be a non-increasing continuous function, and let $z=z_{\psi}$ be the function associated to $\psi$ by Lemma \ref{KM}. Then $(A,\mb{b})\in \widehat{D}_{m,n}(\psi)$ if and only if $\Del(a_t\Lam_{A,\mb{b}})<z_{\psi}(t)$ for all sufficiently large $t$.
\end{lem}

\begin{rem}\label{Alt}
In other words, Lemma \ref{Dani} means that 
$$\widehat{D}_{m,n}(\psi)^c=\limsup_{t\to\infty}\set{(A,\mb{b}): \Del(a_t\Lam_{A,\mb{b}})\geq z_{\psi}(t)},$$ $$\widehat{D}^{\mb{b}}_{m,n}(\psi)^c=\limsup_{t\to\infty}\set{A: \Del(a_t\Lam_{A,\mb{b}})\geq z_{\psi}(t)}.$$
Here, the limsup sets are taken for real values $t\in\bR$. However, in the proof of the convergent part, we are going to work with limsup sets taken for $t\in\bN$ to apply the Hausdorff-Cantelli lemma. Thus, in the Section 3, we will use the following alternative: there exists a constant $C_0>0$ satisfying
$$\widehat{D}_{m,n}(\psi)^c\subseteq\limsup_{t\to\infty, t\in\bN}\set{(A,\mb{b}): \Del(a_t\Lam_{A,\mb{b}})\geq z_{\psi}(t)-C_0},$$ $$\widehat{D}^{\mb{b}}_{m,n}(\psi)^c\subseteq\limsup_{t\to\infty,t\in\bN}\set{A: \Del(a_t\Lam_{A,\mb{b}})\geq z_{\psi}(t)-C_0}.$$
This alternative holds since $z_{\psi}$ is uniformly continuous by Lemma \ref{KM} and $\Del$ is uniformly continuous on the set $\Del^{-1}([z,\infty])$ for any $z\in\bR$ (\cite[Lemma 2.1]{KW19}).
\end{rem}

\begin{lem}\label{dualfg}
Let $\psi:[T_0,\infty)\to\R_{+}$ be a non-increasing continuous function, and let $z=z_{\psi}$ be the function associated to $\psi$ by Lemma \ref{KM}. For $0\leq s\leq mn$, we have 
\eq{
\sum_{q=\lceil T_0\rceil}^{\infty}\frac{1}{\psi(q)q^2}\left(\frac{q^{\frac{1}{n}}}{\psi(q)^{\frac{1}{m}}}\right)^{mn-s}<\infty\iff
\sum_{t=\lceil t_0\rceil}^{\infty}e^{-(m+n)\left(z(t)-\frac{mn-s}{mn}t\right)}<\infty
.}
\end{lem}
\begin{proof}
Note that if $0\leq s\leq mn-n$, both of the sum is infinity regardless of $\psi$, thus we may assume $mn-n< s\leq mn$. 
Following \cite{KM99} and \cite{KW19}, we replace the sums with integrals
\eq{
\int_{T_0}^{\infty}\frac{1}{\psi(x)x^2}\left(\frac{x^{\frac{1}{n}}}{\psi(x)^{\frac{1}{m}}}\right)^{mn-s}dx\quad\text{and}\quad
\int_{t_0}^{\infty}e^{-(m+n)\left(z(t)-\frac{mn-s}{mn}t\right)}dt
}
respectively. Define
\eq{
P:=-\log\circ\psi\circ\exp:[T_0,\infty)\to\R\quad\text{and}\quad\lam(t):=t+nz(t).
}
Since $\psi(e^{\lam})=e^{-P(\lam)}$, we have
\begin{align*}
\int_{T_0}^{\infty}\frac{1}{\psi(x)x^2}\left(\frac{x^{\frac{1}{n}}}{\psi(x)^{\frac{1}{m}}}\right)^{mn-s}dx 
&=\int_{\log{T_0}}^{\infty}e^{-\left(1-\frac{mn-s}{n}\right)\lam+\left(1+\frac{mn-s}{m}\right)P(\lam)}d\lam
\end{align*}
Using $P(\lam(t))=t-mz(t)$, we also have
\begin{align*}
\int_{t_0}^{\infty}e^{-(m+n)\left(z(t)-\frac{mn-s}{mn}t\right)}dt
&=\int_{\log{T_0}}^{\infty}e^{-\left(1-\frac{mn-s}{n}\right)\lam+\left(1+\frac{mn-s}{m}\right)P(\lam)}d\left[\frac{m}{m+n}\lam+\frac{n}{m+n}P(\lam)\right]\\
&=\frac{m}{m+n}\int_{\log{T_0}}^{\infty}e^{-\left(1-\frac{mn-s}{n}\right)\lam+\left(1+\frac{mn-s}{m}\right)P(\lam)}d\lam\\
&+\frac{n}{m+n}\int_{\log{T_0}}^{\infty}e^{-\left(1-\frac{mn-s}{n}\right)\lam+\left(1+\frac{mn-s}{m}\right)P(\lam)}dP(\lam).
\end{align*}
The second term in the last line can be expressed by
\begin{align*}
&\frac{n}{m+n}\int_{\log{T_0}}^{\infty}e^{-\left(1-\frac{mn-s}{n}\right)\lam+\left(1+\frac{mn-s}{m}\right)P(\lam)}dP(\lam)\\&=\frac{n}{m+n}\left(1+\frac{mn-s}{m}\right)^{-1}\int_{\log{T_0}}^{\infty}e^{-\left(1-\frac{mn-s}{n}\right)\lam}d(e^{\left(1+\frac{mn-s}{m}\right)P(\lam)})
\end{align*}
Using integration by parts, the last integral is
\begin{align*}
&\int_{\log{T_0}}^{\infty}e^{-\left(1-\frac{mn-s}{n}\right)\lam}d(e^{\left(1+\frac{mn-s}{m}\right)P(\lam)})\\
&=\left(1-\frac{mn-s}{n}\right)\int_{\log{T_0}}^{\infty}e^{-\left(1-\frac{mn-s}{n}\right)\lam+\left(1+\frac{mn-s}{m}\right)P(\lam)}d\lam\\
&+\left(\displaystyle\lim_{\lam\to\infty}e^{-\left(1-\frac{mn-s}{n}\right)\lam+\left(1+\frac{mn-s}{m}\right)P(\lam)}-T_0^{-\left(1-\frac{mn-s}{n}\right)}\psi(T_0)^{-\left(1+\frac{mn-s}{m}\right)}\right).
\end{align*}
Note that $\displaystyle\lim_{\lam\to\infty}e^{-\left(1-\frac{mn-s}{n}\right)\lam+\left(1+\frac{mn-s}{m}\right)P(\lam)}=0$ if the integral 
\eq{
\int_{\log{T_0}}^{\infty}e^{-\left(1-\frac{mn-s}{n}\right)\lam+\left(1+\frac{mn-s}{m}\right)P(\lam)}d\lam
} converges. Thus the convergence of 
\eq{
\int_{T_0}^{\infty}\frac{1}{\psi(x)x^2}\left(\frac{x^{\frac{1}{n}}}{\psi(x)^{\frac{1}{m}}}\right)^{mn-s}dx \text{ or } \int_{t_0}^{\infty}e^{-(m+n)\left(z(t)-\frac{mn-s}{mn}t\right)}dt
} implies the convergence of the other one since all summands are positive except the finite value $-T_0^{-\left(1-\frac{mn-s}{n}\right)}\psi(T_0)^{-\left(1+\frac{mn-s}{m}\right)}$.

\end{proof}

\subsection{Ubiquitous systems}\label{subsec2.3}
The proof of the divergent parts of Theorem \ref{thm2}, that is the singly metric case, is based on the ubiquity framework developed in \cite{BDV06, BV09}. The concept of ubiquitous systems goes back to \cite{BS70} and \cite{DRV90} as a method of determining lower bounds for the Hausdorff dimension of limsup sets. This concept was developed by Beresnevich, Dickinson and Velani in \cite{BDV06} to provide a very general and abstract approach for establishing the Hausdorff measure of a large class of limsup sets. In this subsection, we introduce a simplified form of ubiquitous systems to deal with the specific application as in \cite[Section 12.1]{BDV06}.

We consider $[0,1]^{mn}$ with the supremum norm $\|\cdot\|$.
Let $\cR:= (R_{\al})_{\al\in J}$ be a family of \textit{resonant sets} $R_{\al}\subset [0,1]^{mn}$ indexed by a countable set $J$. We assume that each resonant set $R_{\al}$ is an $(m-1)n$-dimensional, rational hyperplane following \cite[Section 12.1]{BDV06}.
Let $\beta : J \to \bR^{+}: \al \mapsto \beta_{\al}$ be a positive function on $J$ for which the number of $\al\in J$ with $\beta_{\al}$ bounded above is always finite. Given a set $S\subset [0,1]^{mn}$, let 
\eq{
\Del(S,r):=\{X\in [0,1]^{mn} : \text{dist}(X,S)<r\},
} where $\text{dist}(X,S):=\inf\{\|X-Y\|:Y\in S\}$. 
Fix a decreasing function $\Psi :\bR^{+}\to\bR^{+}$, which is called the \textit{approximating function}. For $N\in\bN$, let
\eq{
\Del(\Psi,N):= \bigcup_{\al\in J\ :\ 2^{N-1}<\beta_{\al}\leq 2^{N} } \Del(R_{\al},\Psi(\beta_{\al}))
} and let
\eq{
\Lam(\Psi):=\limsup_{N\to\infty}\Del(\Psi,N)=\bigcap_{M=1}^{\infty}\bigcup_{N=M}^{\infty}\Del(\Psi,N).
}

Throughout, $\rho : \bR^{+}\to\bR^{+}$ will denote a function satisfying $\lim_{t\to\infty}\rho(t)=0$ and is usually referred to as the \textit{ubiquitous function}. Let
\eq{
\Del(\rho,N):= \bigcup_{\al\in J\ :\ 2^{N-1}<\beta_{\al}\leq 2^{N} } \Del(R_{\al},\rho(\beta_{\al})).
}

\begin{defi}[Local ubiquity]
Let $B$ be an arbitrary ball in $[0,1]^{mn}$. Suppose that there exist a ubiquitous function $\rho$ and an absolute constant $\ka>0$ such that
\eqlabel{LoUb}{
|B\cap \Del(\rho,N)| \geq \ka |B|\quad\text{for } N\geq N_{0}(B),
} where $|\cdot|$ denotes the Lebesgue measure on $[0,1]^{mn}$.
Then the pair $(\cR,\beta)$ is said to be a \textit{locally ubiquitous} system relative to $\rho$. 
\end{defi}

With notations in \cite{BDV06}, the Lebesgue measure on $[0,1]^{mn}$ is of type (M2) with $\de=mn$ and the intersection conditions are also satisfied with $\ga=(m-1)n$ (see \cite[Section 12.1]{BDV06}). These conditions are not stated here but these extra conditions exist and need to be established for the more abstract ubiquity. 

Finally, a function $h$ is said  to be \textit{2-regular} if there exists a positive constant $\lam<1$ such that for $N$ sufficiently large 
\eq{
h(2^{N+1})\leq \lam h(2^{N}).
}

The following theorem is a simplified version of \cite[Theorem 1]{BV09}.
\begin{thm}\cite[Theorem 1]{BV09}\label{UbThm1}
Suppose that $(\cR,\beta)$ is a local ubiquitous system relative to $\rho$ and that $\Psi$ is an approximating function. Furthermore, suppose that $\rho$ is $2$-regular. Then for $(m-1)n<s\leq mn$
\eq{
\cH^{s}(\Lam(\Psi))=\cH^{s}([0,1]^{mn})\quad \text{if}\quad \sum_{N=1}^{\infty}\frac{\Psi(2^N)^{s-(m-1)n}}{\rho(2^N)^{n}}=\infty.
}
\end{thm}

\section{Proof of the convergent part}
\subsection{Weak-$L^1$ estimate}
To obtain the upper bound of Hausdorff dimension, we will basically count the number of covering balls following the ideas from \cite{KKLM}. We are going to use the equidistribution of expanding subgroup of the $a_t$-action on $X$ to compute the Lebesgue measure of the set of points visiting the shrinking target for each time $t$, following the ``thickening" technique of Margulis \cite{M04}. We also refer to the formulation of \cite{KM96}. However, if we apply the thickening argument for $L^2$ functions as usual, it does not give the optimal dimension upper bound. To obtain the optimal dimension bound, we need a $L^{1,w}$ estimate as the following Proposition \ref{Lw1}. $L^{1,w}$ norm of a function $f$ on $X$ is defined by $||f||_{L^{1,w}(X)}:=\displaystyle\sup_{M>0}Mm_X(\set{x\in X: |f(x)|\ge M})$, and $L^{1,w}(X)$ is the space of measurable functions with finite $L^{1,w}$-norm.
\begin{prop}\label{Lw1}
Let $H, \tilde{H}$ be the maximal expanding, nonexpanding subgroup of $G_d$, respectively. Assume that $f\in L^{1,w}(X)$ is a nonnegative function satisfying the following condition: there exist $c, r_0>0$ such that $c<|\frac{f(\tilde{h}x)}{f(x)}|$ for any $\tilde{h}\in B_{r_0}^{\tilde{H}}(id)$, $x\in X$. Then for any $x\in X$, there exist a constant $K=K(x)>0$ such that
$$m_H(\set{h\in B_1^H(id): f(a_t hx)\ge M})\le \frac{K}{M}$$
for all $M>0$, $t>0$, i.e. for any $x\in X$, $||(a_t)_*f_x||_{L^{1,w}(B_1^H(id))}$ is uniformly bounded for all $t>0$, where the function $f_x:H\rightarrow \bR$ is defined by $f_{x}(h)=f(hx)$.
\end{prop}
\begin{proof}
Fix $x\in X$ and let $E_{M,t}:=\set{h\in B_1^H(id): f(a_t hx)\ge M}$. For contradiction, suppose that for any $K>0$, there exist $t,M>0$ such that $m_H(E_{M,t})>\frac{K}{M}$. Let $\hat{E}_{M,t}:=\set{\tilde{h}h: h\in E_{M,t}, \tilde{h}\in B_r^{\tilde{H}}(id)}$, where $0<r<r_0$ is a small real number to be determined later. Then for any $\tilde{h}h\in\hat{E}_{M,t}$,
$$f(a_t \tilde{h}hx)=f((a_t\tilde{h}a_t^{-1})a_t hx)>cf(a_t hx)\ge cM$$ since $a_t\tilde{h}a_t^{-1}\in B_r^{\tilde{H}}(id)$. We partition $B^H_1(id)$ into $D_1,\cdots D_N$ so that a map $\pi_x: G\rightarrow X$ defined by $\pi_x(g)=gx$ is injective on each $D_i$. Note that the number of the partition $N$ is not depending on $K$. Choose $r$ small enough so that $\pi_x$ is injective on $B^{\tilde{H}}_r(id)D_i$ for all $1\leq i\leq N$. Let $E_{M,t}=\displaystyle\bigsqcup_{i=1}^{N}E_i$, where $E_i=E_{M,t}\cap D_i$, then
\begin{align*}
m_X(\set{y\in X: f(y)\ge cM})&=
m_X(\set{y\in X: f(a_t y)\ge cM})
\\&\ge m_{X}(\set{\tilde{h}hx\in X: \tilde{h}h\in \hat{E}_{M,t}})
\\&\ge m_{G}(\set{\tilde{h}h\in G: \tilde{h}\in B^{\tilde{H}}_r(id), h\in E_i})
\\&\asymp m_{\tilde{H}}(B_r^{\tilde{H}}(id))m_H(E_i)
\end{align*} for all $1\leq i\leq N$. Summing over $1\leq i\leq N$, we have
\begin{align*}
Nm_X(\set{y\in X: f(y)\ge cM})&\gg m_{\tilde{H}}(B_r^{\tilde{H}}(id))m_H(E_{M,t})\\
&>\frac{m_{\tilde{H}}(B_r^{\tilde{H}}(id))K}{M}
\end{align*}
and it implies $\|f\|_{L^{1,w}(X)}=\infty$ since $K>0$ is arbitrary and $c,r,N$ are independent to $K$. It contradicts the assumption $f\in L^{1,w}(X)$.
\end{proof}
\subsection{successive minima function} 
Let $\lambda_j(\Lambda)$ denote the $j$-th successive minimum of a lattice $\Lambda\subseteq \bR^d$ i.e. the infimum of $\lambda$ such that the ball $B_\lambda^{\bR^d}(0)$ contains $j$ independent vectors of $\Lambda$. The following inequality explains the relationship between the successive minima functions $\lambda_1$ and $\lambda_d$.
\begin{thm}[Mahler's inequality, \cite{Cas}, Theorem \rom{6} in Chapter \rom{8}]
For any lattice $\Lambda\subseteq \bR^d$, $1\leq\lambda_1(\Lambda^{*})\lambda_d(\Lambda)\leq d!$ holds, where $\Lambda^{*}$ is the dual lattice of $\Lambda$.
\end{thm}
Note that the Haar measure $m_X$ is invariant under the dual operation since the dual operation is induced by the transpose of the inverse of a matrix, which is an automorphism of $G$. Another ingredient we will use is Siegel's integral formula.
\begin{thm}[Siegel's integral formula]
For a compactly supported integrable function $f\in L^1(\bR^d)$, we define a function $\hat{f}$ on $X$ by
$$\hat{f}(\Lambda)=\displaystyle\sum_{v\in\Lambda\setminus\set{0}}f(v).$$
Then for any $f$ as above, $\int_{X}\hat{f}dm_X=\int_{\bR^d}fdm_{\bR^d}$.
\end{thm}

In the following Proposition \ref{lamLw1} and \ref{lamuni}, we will show that the function $\lambda_d^d$ satisfies the assumption of Proposition \ref{Lw1}.

\begin{prop}\label{lamLw1}
$\lambda_d^d\in L^{1,w}(X).$
\end{prop}
\begin{proof}
For any $r>0$,
\eqlabel{eqn1}{\begin{aligned}
m_X(\set{\Lambda: \lambda_d^d(\Lambda)\ge (d!)^dr^{-d}})
&=m_X(\set{\Lambda: \lambda_d(\Lambda)\ge d!r^{-1}})\\
&\leq m_X(\set{\Lambda: \lambda_1(\Lambda^{*})\leq r})\\
&=m_X(\set{\Lambda: \lambda_1(\Lambda)\leq r})\\
&\leq\displaystyle\int_{\Lambda\in X}\widehat{\chi_{B_r(0)}}(\Lambda)dm_X(\Lambda)\\
&=\displaystyle\int_{\bR^d}\chi_{B_r(0)}dm_{\bR^d}\asymp r^d,
\end{aligned}}
thus $\lambda_d^d\in L^{1,w}(X)$.
In (\ref{eqn1}), the second line is by Mahler's inequality, the third line is by the invariance of $m_X$ under the dual operation, the fourth line is using the fact that $\lambda_1(\Lambda)\leq r$ implies $\widehat{\chi_{B_r(0)}}(\Lambda)\ge 1$, and the last line is by Siegel's integral formula.
\end{proof}

\begin{prop}\label{lamuni}
For any $0<c<1$, there exists $r>0$ such that for any
$g\in G$ with $\mb{d}(g,id)<r$, $c\lambda_d(\Lambda)<\lambda_d(g\Lambda)<c^{-1}\lambda_d(\Lambda)$ holds for any $\Lambda\in X$.
\end{prop}
\begin{proof}
It suffices to show the statement under the stronger assumption that both of $g$ and $g^{-1}$ are in the ball $B_r^{G}(id)$. Then there exist independent vectors $v_1,\cdots,v_d \in\Lambda$ such that $\|v_1\|\leq\|v_2\|\leq\cdots\leq\|v_d\|=\lambda_d(\Lambda)$. For each $1\leq i\leq d$,
$$\|gv_i-v_i\|\leq d||g-id||\|v_i\|\leq dr\lambda_d(\Lambda),$$ thus $\|gv_i\|\leq (1+dr)\lambda_d(\Lambda)$. It implies $\lambda_d(g\Lambda)\leq (1+dr)\lambda_d(\Lambda)$ since $gv_1,\cdots,gv_d$ are independent vectors. Applying this for $g^{-1}$ and $g\Lambda$, instead of $g$ and $\Lambda$, we have
$$\lambda_d(\Lambda)=\lambda_d(g^{-1}g\Lambda)\leq(1+dr)\lambda_d(g\Lambda).$$
Thus for any $\Lambda\in X$ and $g\in B_r^{G}(id)$, $(1+dr)^{-1}\lambda_d(\Lambda)<\lambda_d(g\Lambda)<(1+dr)\lambda_d(\Lambda)$ holds.
\end{proof}

\subsection{The number of covering balls}
In this subsection, we will construct a sequence of coverings for $\widehat{D}_{m,n}(\psi)^c$ and $\widehat{D}^{\mb{b}}_{m,n}(\psi)^c$ to apply Hausdorff-Cantelli Theorem. Recall that we adopt the supremum norm $\|\cdot\|$ on $[0,1]^{mn}$.
\begin{prop}\label{covering}
Let $C_0>0$ be a constant described in Remark \ref{Alt}. For $t\in\bN$, let $Z_t:=\set{A\in [0,1]^{mn}: \log(d\lambda_d(a_t\Lambda_{A}))\ge z_{\psi}(t)-C_0}$. Then $Z_t$ can be covered with $Ke^{(m+n)(t-z_{\psi}(t))}$ balls in $M_{m,n}(\bR)$ of radius $\frac{1}{2}e^{-(\frac{1}{m}+\frac{1}{n})t}$ for a constant $K>0$ depending only on the dimension $d$.
\end{prop}
\begin{proof}
$[0,1]^{mn}$ can be covered with $p$($\asymp e^{(m+n)t}$) cubes $D_1,D_2,\cdots,D_p$ with sides parallel to the axes of $\bR^{mn}$ and of sidelength $r\leq e^{-(\frac{1}{m}+\frac{1}{n})t}$ and having mutually disjoint interiors.

\begin{lem}\label{lem1}
For $t\in\bN$, let $Z'_t:=\set{A\in [0,1]^{mn}: \log(d^2\lambda_d(a_t\Lambda_{A}))\ge z_{\psi(t)}-C_0-1}$. For any $t\geq 1$, if $D_i\cap Z_t\neq\phi$ for some $1\leq i\leq p$, then $D_i\subset Z'_t$.
\end{lem}
\begin{proof}
Assume that there exists $x\in D_i$ but $x\notin Z'_t$ for some $t>0$. Choose a point $y\in D_i\cap Z_t$, then $\|x-y\|\leq r$ and
\begin{align*}
    ||a_tu_{x-y}a_{-t}-id||
    &=\left |\left | \left(\begin{matrix}
I_m & e^{(\frac{1}{m}+\frac{1}{n})t}(x-y)\\
 & I_n\\
\end{matrix}\right)-id\right |\right |\\
&\leq\|e^{(\frac{1}{m}+\frac{1}{n})t}(x-y)\|\leq 1.
\end{align*}
Thus, for $g=a_tu_{x-y}a_{-t}$, it satisfies $||g-id||\leq 1$ and $a_t\Lambda_{y}=ga_t\Lambda_{x}$.
On the other hand, $\log(d^{2}\lambda_d(a_t\Lambda_{x}))<z_\psi(t)-C_0-1$, $\log(d\lambda_d(a_t\Lambda_{y}))\ge z_\psi(t)-C_0$ hold since $x\notin Z'_t$, $y\in Z_t$. We can take independent vectors $v_1,\cdots,v_d\in\bR^d$ in the lattice $a_t\Lambda_{x}$ satisfying $\|v_i\|<\frac{1}{d^2}e^{z_\psi(t)-C_0-1}$ for all $1\leq i\leq d$. Let $w_i=gv_i$, then $w_i$'s are independent vectors in the lattice $a_t\Lambda_{y}$ and satisfy
$$\|w_i\|\leq d||g||\|v_i\|\leq 2d\|v_i\|<\frac{2}{d}e^{z_\psi(t)-C_0-1}<\frac{1}{d}e^{z_\psi(t)-C_0}$$
for all $1\leq i\leq d$.
Thus we obtain $\log(d\lambda_d(a_t\Lambda_{y}))< z_\psi(t)-C_0$ but it contradicts to $y\in Z_t$.
\end{proof}
Let $p':=|\set{D_i: D_i\cap Z_t\neq\phi}|$ and by reordering the $D_i$'s if necessary, we can assume that $\set{D_1,\cdots,D_{p'}}=\set{D_i: D_i\cap Z_t\neq\phi}$. Then $Z_t\subseteq\displaystyle\bigcup_{i=1}^{p'}D_i\subseteq Z'_t$ by Lemma \ref{lem1}. Now we will apply Proposition \ref{Lw1} for the function $\lambda_d^d$ with the base point $x=\bZ^d$. By Proposition \ref{lamLw1} and Proposition \ref{lamuni}, $\lambda_d^d$ satisfies the conditions of Proposition \ref{Lw1}. Then we have
\begin{align*}
  m_{\bR^{mn}}(Z'_t)&\leq m_{\bR^{mn}}\left(\set{A\in [0,1]^{mn}: \lambda_d(a_t\Lambda_A)\ge \frac{1}{d^2}e^{z_\psi(t)-C_0-1}}\right)\\
  &\asymp m_H\left(\set{h\in B_{1}^{H}(id): \lambda_d^d(a_t h\bZ^d)\ge \frac{1}{d^{2d}}e^{d(z_\psi(t)-C_0-1)}}\right)\\
  &\ll e^{-dz_\psi(t)}.
\end{align*}
On the other hand, $m_{\bR^{mn}}(Z'_t)\ge\displaystyle\sum_{i=1}^{p'}m_{\bR^{mn}}(D_i)=p'e^{-dt}$ holds, thus we finally obtain $p'\ll e^{d(t-z_\psi(t))}$. It means that $Z_t$ can be covered by $\ll e^{d(t-z_\psi(t))}$ many balls of $r$-radius since $Z_t\subseteq\displaystyle\bigcup_{i=1}^{p'}D_i$.
\end{proof}

\begin{prop}\label{Ztdim}
Let $0\leq s\leq mn$. If $\sum_{q= 1}^{\infty}\frac{1}{\psi(q)q^2}\left(\frac{q^{\frac{1}{n}}}{\psi(q)^{\frac{1}{m}}}\right)^{mn-s}<\infty$, then $\cH^{s}(\displaystyle\limsup_{t\to\infty} Z_t)=0$ and $\cH^{s+m}(\displaystyle\limsup_{t\to\infty} Z_t\times[0,1]^m)=0$.
\end{prop}
\begin{proof}
By Lemma \ref{dualfg}, the assumption $\sum_{q= 1}^{\infty}\frac{1}{\psi(q)q^2}\left(\frac{q^{\frac{1}{n}}}{\psi(q)^{\frac{1}{m}}}\right)^{mn-s}<\infty$ is equivalent to $\sum_{t=1}^{\infty}e^{-(m+n)\left(z(t)-\frac{mn-s}{mn}t\right)}<\infty$. For each $t\in\bN$, let $D_{t1},D_{t2},\cdots, D_{tp_t}$ be the balls of radius $\frac{1}{2}{e^{-(\frac{1}{m}+\frac{1}{n})t}}$ covering $Z_t$ as in Proposition \ref{covering}. Note that $p_t$, the number of the balls, is not greater than $Ke^{(m+n)(t-z_{\psi}(t))}$ by Proposition \ref{covering}. By applying Lemma \ref{hcan} to a sequence of balls $\set{D_{tj}}_{t\in\bN, 1\leq j\leq p_t}$, we have $\cH^{s}(\displaystyle\limsup_{t\to\infty} Z_t)\leq \cH^{s}(\displaystyle\limsup_{t\to\infty} D_{tj})=0$. 

We prove the second statement by a similar argument. Proposition \ref{covering} implies that $Z_t\times [0,1]^m$ can be covered with $Ke^{\frac{m+n}{n}t}e^{(m+n)(t-z_{\psi}(t))}$ balls of radius $\frac{1}{2}{e^{-(\frac{1}{m}+\frac{1}{n})t}}$. Applying Lemma \ref{hcan} again, we 
have $\cH^{s+m}(\displaystyle\limsup_{t\to\infty} Z_t\times[0,1]^m)=0$.

\end{proof}
The convergent part of Theorem \ref{thm1} and \ref{thm2} follows this proposition.
\begin{proof}[Proof of Theorem \ref{thm1} and \ref{thm2}]
We first prove the singly metric case, Theorem \ref{thm2}. We claim that $\log(d\lambda_d(a_t\Lambda_A))\ge\Delta(a_t\Lambda_{A,\mb{b}})$ for every $\mb{b}\in\bR^m$. Let $v_1,\cdots,v_d$ be independent vectors satisfying $\|v_i\|\leq\lambda_d(a_t\Lambda_A)$ for $1\leq i\leq d$. Then there exists a vector of $a_t\Lambda_{A,\mb{b}}$ which can be written as a form of $\displaystyle\sum_{i=1}^{d}\alpha_iv_i$ for some $-1\leq \alpha_i\leq 1$'s, so the length of the shortest vector is $\leq\displaystyle\sum_{i=1}^{d}\|v_i\|$. Thus, $\Delta(a_t\Lambda_{A,\mb{b}})\leq\log\displaystyle\sum_{i=1}^{d}\|v_i\|\leq \log(d\lambda_d(a_t\Lambda_A))$. It implies $\widehat{D}_{m,n}^{\mb{b}}(\psi)^c\subseteq\displaystyle\limsup_{t\to\infty}\set{A\in [0,1]^{mn}: \Delta(a_t\Lambda_{A,\mb{b}})\ge z_\psi(t)-C_0}\subseteq\displaystyle\limsup_{t\to\infty}Z_t$ by Lemma \ref{Dani}, thus we obtain $\cH^{s}(\widehat{D}_{m,n}^{\mb{b}}(\psi)^c)\leq\cH^{s}(\displaystyle\limsup_{t\to\infty} Z_t)=0$ by Proposition \ref{Ztdim}. 

Similarly for the doubly metric case, together with the second statement of Proposition \ref{Ztdim}, $\widehat{D}_{m,n}(\psi)^c\subseteq\displaystyle\limsup_{t\to\infty}\set{(A,\mb{b})\in [0,1]^{mn+m}: \Delta(a_t\Lambda_{A,\mb{b}})\ge z_\psi(t)-C_0}\subseteq\displaystyle\limsup_{t\to\infty}Z_t\times[0,1]^m$ provides the proof of Theorem \ref{thm1}.
\end{proof}

\section{Proof of the divergent part}

\subsection{Notation and a transference lemma}\label{subsec4.1}
Let $d=m+n$ and assume that $\psi:[T_0,\infty)\to\R_{+}$ be a decreasing function satisfying $\lim_{T\to\infty}\psi(T)=0$. Denote by $\|\cdot\|_{\bZ}$ and $|\cdot|_{\bZ}$ the distance to the nearest integral vector and integer, respectively. Define the function $\widetilde{\psi} : [S_0,\infty)\to\R_{+}$ by \eq{
\widetilde{\psi}(S)=\left(\psi^{-1}(S^{-m})\right)^{-\frac{1}{n}},
} where $S_0 = \psi(T_0)^{-1/m}$.
We associate $\psi$-Dirichlet non-improvability with $\widetilde{\psi}$-approximability via a transference lemma as follows.

\begin{lem}[A transference lemma, \cite{Cas57}]\label{dual}
Given $(A,\mb{b})\in \widetilde{M}_{m,n}(\bR)$, if the system
\eq{
\| {^{t}A}\mb{x}\|_{\bZ}<d^{-1}|\mb{b}\cdot\mb{x}|_{\bZ}\widetilde{\psi}(S)\quad\text{and}\quad \|\mb{x}\|<d^{-1}|\mb{b}\cdot\mb{x}|_{\bZ}S
} has a nontrivial solution $\mb{x}\in\bZ^m$ for an unbounded set of $S\geq S_0$, then $(A,\mb{b})\in \widehat{D}_{m,n}(\psi)^{c}$. 
\end{lem}
\begin{proof}
Using part A of Theorem \rom{17} in Chapter \rom{5} of \cite{Cas57} with $C=\psi(T)^{1/m}$ and $X=T^{1/n}$, the fact that 
\eq{
\|A\mb{q}-\mb{b}\|_{\bZ}\leq \psi(T)^{1/m}\quad\text{and}\quad \|\mb{q}\|\leq T^{1/n}
} for some $\mb{q}\in\bZ^n$ implies that 
\eq{
|\mb{b}\cdot\mb{x}|_{\bZ}\leq d\max(T^{1/n}\|^t A\mb{x}\|_{\bZ}, \psi(T)^{1/m}\|\mb{x}\|)
} holds for all $\mb{x}\in\bZ^m$. Thus the lemma follows with $S=\psi(T)^{-1/m}$ and $\widetilde{\psi}(S)=T^{-1/n}$ since $\lim_{T\to\infty}\psi(T)=0$.
\end{proof}

Thus we adopt the following notations for each $S\geq S_0$ and $0<\eps<1/2$ :
\begin{itemize}
	\item Let $W_{S,\eps}$ be the set of $A \in [0,1]^{mn}$ such that there exists $\mb{x}_{A,S} \in\bZ^m \setminus \{\mb{0}\}$ satisfying
\eq{
\| {^{t}A}\mb{x}_{A,S}\|_{\bZ}<d^{-1}\eps\widetilde{\psi}(S)\quad\text{and}\quad \|\mb{x}_{A,S}\|<d^{-1}\eps S.
}
	\item $\widehat{W}_{S,\eps}:= \{(A,\mb{b})\in [0,1]^{mn+m}: A\in W_{S,\eps}\text{ and } |\mb{b}\cdot\mb{x}_{A,S}|_{\bZ}>\eps\}$.

	\item For fixed $\mb{b}\in\bR^m$, let $W_{\mb{b},S,\eps}$ be the set of $A \in [0,1]^{mn}$ such that there exists $\mb{x} \in\bZ^m \setminus \{\mb{0}\}$ satisfying 
\begin{enumerate} [label=(\roman*)]
	\item $|\mb{b}\cdot\mb{x}|_{\bZ}>\eps$,
	\item $\| {^{t}A}\mb{x}\|_{\bZ}<d^{-1}\eps\widetilde{\psi}(S)\quad\text{and}\quad \|\mb{x}\|<d^{-1}\eps S.$
\end{enumerate}
	\item $W_{\mb{b},\eps}:=\displaystyle\limsup_{S\to\infty}W_{\mb{b},S,\eps}$.
\end{itemize}

Note that $A\in W_{S,\eps}$ if and only if 
\eq{
\| {^{t}A}\mb{x}_{A,S}\|_{\bZ}<\Psi_{\eps}(U) \quad\text{and}\quad \|\mb{x}_{A,S}\|<U,
} 
where 
\eqlabel{psieps}
{
\Psi_{\eps}(U) := d^{-1}\eps\widetilde{\psi}(d \eps^{-1}U).
}
By Lemma \ref{dual}, $\displaystyle\limsup_{S\to\infty}\widehat{W}_{S,\eps} \subset \widehat{D}_{m,n}(\psi)^{c}$ and $W_{\mb{b},\eps} \subset \widehat{D}_{m,n}^{\mb{b}}(\psi)^{c}$.

We remark that $\displaystyle\limsup_{S\to\infty}W_{S,\eps}$ is the set of $\Psi_{\eps}$-approximable matrices, that is, $\displaystyle\limsup_{S\to\infty}W_{S,\eps}=\{A\in[0,1]^{mn}: {^{t}A}\in W_{n,m}(\Psi_{\eps})\}$. Here and hereafter, as mentioned before in footnote \ref{fn1}, we adopt the slightly different definition for $\Psi_{\eps}$-approximability, where the inequality $\|{^{t}A}\mb{x}\|_{\bZ} < \Psi_{\eps}(\|\mb{x}\|)$ is used instead of \eqref{approx}. Then, $W_{\mb{b},\eps}$ can be considered as the set of $\Psi_{\eps}$-approximable matrices with solutions restricted on the set $\{\mb{x}\in\bZ^{m}:|\mb{b}\cdot\mb{x}|_{\bZ}>\eps\}$.

\subsection{Mass distributions on $\Psi_{\eps}$-approximable matrices}
In this subsection, we prove the divergent part of Theorem \ref{thm1} using mass distributions on $\Psi_{\eps}$-approximable matrices following \cite{AB18}.

\begin{lem}\label{dualsum}
For each $0\leq s\leq mn$ and $0<\eps<1/2$, let $U_0 = d^{-1}\eps S_0$. Then,
\eq{
\sum_{q=\lceil T_0\rceil}^{\infty}\frac{1}{\psi(q)q^2}\left(\frac{q^{\frac{1}{n}}}{\psi(q)^{\frac{1}{m}}}\right)^{mn-s}<\infty\iff
\sum_{h=\lceil U_{0}\rceil}^{\infty}h^{m+n-1}\left(\frac{{\Psi_{\eps}}(h)}{h}\right)^{s-n(m-1)}<\infty.
}
\end{lem}
\begin{proof}
Since $\Psi_{\eps}(h) =d^{-1}\eps\widetilde{\psi}(d \eps^{-1}h)$,
\eq{
\sum_{h=\lceil U_{0}\rceil}^{\infty}h^{m+n-1}\left(\frac{\Psi_{\eps}(h)}{h}\right)^{s-n(m-1)}<\infty\iff 
\sum_{q=\lceil S_{0}\rceil}^{\infty}q^{m+n-1}\left(\frac{\widetilde{\psi}(q)}{q}\right)^{s-n(m-1)}<\infty.
} Thus, similar to Lemma \ref{dualfg}, we may assume $mn-n< s \leq mn$ and replace the sums with integrals
\eq{
\int_{T_0}^{\infty}\frac{1}{\psi(x)x^2}\left(\frac{x^{\frac{1}{n}}}{\psi(x)^{\frac{1}{m}}}\right)^{mn-s}dx\quad\text{and}\quad
\int_{S_0}^{\infty}y^{m+n-1}\left(\frac{\widetilde{\psi}(y)}{y}\right)^{s-n(m-1)}dy,
}
respectively. Since $\widetilde{\psi}(y)=\psi^{-1}(y^{-m})^{-\frac{1}{n}}$, we have
\begin{align*}
\int_{S_0}^{\infty}y^{m+n-1}\left(\frac{\widetilde{\psi}(y)}{y}\right)^{s-n(m-1)}dy
&=\int_{S_0}^{\infty}y^{mn+m-1-s}\left(\psi^{-1}(y^{-m})\right)^{m-1-\frac{s}{n}}dy\\
&=\frac{1}{m}\int_{(S_0)^m}^{\infty}t^{n-\frac{s}{m}}\left(\psi^{-1}(t^{-1})\right)^{m-1-\frac{s}{n}}dt\\
&=\frac{1}{m}\int_{\psi^{-1}(S_{0}^{-m})}^{\infty} x^{m-1-\frac{s}{n}}\psi(x)^{-n+\frac{s}{m}}d\psi(x)^{-1}\\
&=\frac{1}{m}\left(n+1-\frac{s}{m}\right)^{-1}\int_{T_0}^{\infty} x^{m-1-\frac{s}{n}}d\left(\psi(x)^{-1}\right)^{n+1-\frac{s}{m}}
\end{align*}
Using integration by parts, 
\begin{align*}
&\int_{T_0}^{\infty} x^{m-1-\frac{s}{n}}d\left(\psi(x)^{-1}\right)^{n+1-\frac{s}{m}}\\
&=\left(\displaystyle\lim_{x\to\infty}x^{m-1-\frac{s}{n}}\psi(x)^{-n-1+\frac{s}{m}}-T_{0}^{m-1-\frac{s}{n}}\psi(T_0)^{-n-1+\frac{s}{m}}\right)\\
&+\left(\frac{s-n(m-1)}{n}\right)\int_{T_0}^{\infty} \psi(x)^{-n-1+\frac{s}{m}}x^{m-2-\frac{s}{n}}dx.
\end{align*}
Observe that 
\eq{
\int_{T_0}^{\infty} \psi(x)^{-n-1+\frac{s}{m}}x^{m-2-\frac{s}{n}}dx=\int_{T_0}^{\infty} \psi(x)^{-n-1+\frac{s}{m}}x^{m-1-\frac{s}{n}}d\log{x}.
}
Thus the convergence of $\int_{T_0}^{\infty} \psi(x)^{-n-1+\frac{s}{m}}x^{m-2-\frac{s}{n}}dx$ gives that \eq{\displaystyle\lim_{x\to\infty}x^{m-1-\frac{s}{n}}\psi(x)^{-n-1+\frac{s}{m}}<\infty.}
Hence the convergence of 
\eq{
\int_{T_0}^{\infty}\frac{1}{\psi(x)x^2}\left(\frac{x^{\frac{1}{n}}}{\psi(x)^{\frac{1}{m}}}\right)^{mn-s}dx\quad\text{or}\quad
\int_{S_0}^{\infty}y^{m+n-1}\left(\frac{\widetilde{\psi}(y)}{y}\right)^{s-n(m-1)}dy
}
implies the convergence of the other one since all summands are positive except the finite value $-T_{0}^{m-1-\frac{s}{n}}\psi(T_0)^{-n-1+\frac{s}{m}}$.
\end{proof}

\begin{lem}\cite[Section 5]{AB18}\label{mass}
Assume that \eq{\sum_{q=1}^{\infty}\frac{1}{\psi(q)q^2}\left(\frac{q^{\frac{1}{n}}}{\psi(q)^{\frac{1}{m}}}\right)^{mn-s}=\infty.} Fix $0<\eps<1/2$. Then, for any $\eta>1$, there exists a probability measure $\mu$ on $\displaystyle\limsup_{S\to\infty} W_{S,\eps}$ satisfying the condition that for any arbitrary ball $D$ of sufficiently small radius $r(D)$ we have 
\eq{
\mu(D)\ll \frac{r(D)^s}{\eta},
}
where the implied constant does not depend on $D$ or $\eta$. 
\end{lem}
\begin{proof} 
Note that $\displaystyle\limsup_{S\to\infty} W_{S,\eps}=\{A\in[0,1]^{mn}: {^{t}A} \in W_{n,m}(\Psi_{\eps})\}$. By Lemma \ref{dualsum}, $\sum_{h=1}^{\infty}h^{m+n-1}\left(\frac{{\Psi_{\eps}}(h)}{h}\right)^{s-n(m-1)}=\infty$, which is the divergent assumption of Jarn\'{i}k's Theorem (Theorem \ref{Jarnik}) for $W_{n,m}(\Psi_{\eps})$. From the proof of Jarn\'{i}k's Theorem in \cite{AB18} and the construction of a probability measure in \cite[Section 5]{AB18} we can obtain a probability measure $\mu$ on $\displaystyle\limsup_{S\to\infty} W_{S,\eps}$ satisfying the above condition.
\end{proof}

Let us give a proof of the divergence part of Theorem \ref{thm1}.

\begin{proof}[Proof of Theorem \ref{thm1}]
If $s=mn+m$, then it follows from Theorem \ref{KW}. Assume that $m\leq s<mn+m$ and fix $0<\eps<1/2$. For any fixed $\eta>1$, let $\mu$ be a probability measure on $\displaystyle\limsup_{S\to\infty} W_{S,\eps}$ as in Lemma \ref{mass} with $s-m$ instead of $s$. Consider the product measure $\nu=\mu\times m_{\bR^m}$, where $m_{\bR^{m}}$ is the canonical Lebesgue measure on $\bR^m$, and let $\pi_{1}$ and $\pi_{2}$ be the natural projections from $\bR^{mn+m}$ to $\bR^{mn}$ and $\bR^{m}$, respectively.  
For any fixed integer $N\geq1$, let $V_{S,\eps}=W_{S,\eps}\setminus \displaystyle\bigcup_{k=N}^{S-1}W_{k,\eps}$ and $\widehat{V}_{S,\eps}=\{(A,\mb{b})\in \widehat{W}_{S,\eps}: A\in V_{S,\eps}\}$ and $E_{A,S,\eps}=\{\mb{b}\in[0,1]^m:|\mb{b}\cdot\mb{x}_{A,S}|_{\bZ}>\eps \}$. 
Note that $m_{\bR^m}(E_{A,S,\eps})\geq1-2\eps$. Using Fubini's theorem,  we have
\begin{align*}
\nu(\bigcup_{S\geq N} \widehat{W}_{S,\eps})&=\nu(\bigcup_{S\geq N} \widehat{V}_{S,\eps})=\sum_{S\geq N}\nu(\widehat{V}_{S,\eps})\\
&\geq \sum_{S\geq N}(1-2\eps)\mu(V_{S,\eps})=(1-2\eps)\mu(\bigcup_{S\geq N} W_{S,\eps})\\
&=1-2\eps.
\end{align*}
Since $N\geq1$ is arbitrary, we have $\nu(\displaystyle\limsup_{S\to\infty}\widehat{W}_{S,\eps})\geq1-2\eps$. 

For any arbitrary ball $B\subset \bR^{mn+m}$ of sufficiently small radius $r(B)$, we have 
\eqlabel{MDEQ}{
\nu(B)=\mu(\pi_1(B))\times m_{\bR^{m}}(\pi_2(B))\ll \frac{r(B)^{s}}{\eta},
}
where the implied constant does not depend on $B$ or $\eta$. If $0 \leq s < m$, we have \eqref{MDEQ} with $\mu$ in Lemma \ref{mass} with $s=0$.

By the Mass Distribution Principle (Lemma \ref{MDP}) and Lemma \ref{dual}, we have 
\eq{
\cH^{s}(\widehat{D}_{m,n}(\psi)^{c})\geq\cH^{s}(\displaystyle\limsup_{S\to\infty}\widehat{W}_{S,\eps})\gg (1-2\eps)\eta
}
and the proof is finished by taking $\eta\to\infty$.

\end{proof}

\subsection{Local ubiquity for $W_{\mb{b},\eps}$} 
The singly metric case is more complicated than the doubly metric case. In this subsection, we will prove Theorem \ref{thm2} by establishing the ubiquitous system for $W_{\mb{b},\eps}$ with an appropriate $\eps$ as follows.

For $\mb{b}=(b_1,\dots,b_m)\in \bR^{m}\setminus\bZ^{m}$, define 
\eqlabel{beps}{
\eps(\mb{b}):=\min_{1\leq j\leq m,\ |b_{j}|_{\bZ}>0} \frac{|b_{j}|_{\bZ}}{4}.
}
Note that the fact that $\mb{b}\in \bR^{m}\setminus\bZ^{m}$ implies $\eps(\mb{b})>0$.
The following lemma will be used when we count the number of integral vectors $\mb{z}\in\bZ^{m}$ such that 
\eqlabel{BadP}{
|\mb{b}\cdot\mb{z}|_{\bZ} \leq \eps(\mb{b}). 
}

\begin{lem}\label{Count}
For $\mb{b}=(b_1,\dots,b_m)\in \bR^m\setminus\bZ^m$, let $\eps(\mb{b})$ be as in \eqref{beps} and $1\leq i \leq m$ be an index such that $\eps(\mb{b})=\frac{|b_{i}|_{\bZ}}{4}$.
Then, for any $\mb{x}\in\bZ^{m}$, at most one of $\mb{x}$ and $\mb{x}+\mb{e}_{i}$ satisfies \eqref{BadP},
where $\mb{e}_{i}$ denotes the vector with a $1$ in the $i$th coordinate and $0$'s elsewhere.
\end{lem}
\begin{proof}
Observe that if $|\mb{b}\cdot\mb{x}|_{\bZ}\leq\eps(\mb{b})$, then
\eq{
\big| |\mb{b}\cdot(\mb{x}\pm\mb{e}_{i})|_{\bZ} - |\pm b_{i}|_{\bZ}\big| \leq |\mb{b}\cdot\mb{x}|_{\bZ} \leq \eps(\mb{b}).
} By definition of $\eps(\mb{b})$, we have
\eq{
 |\mb{b}\cdot(\mb{x}\pm\mb{e}_{i})|_{\bZ} \geq |b_{i}|_{\bZ} - \eps(\mb{b}) > \eps(\mb{b}). 
}
\end{proof}

Now we fix $\mb{b}\in\bR^{m}\setminus\bZ^{m}$ and write $\eps_{0}:=\eps(\mb{b})$ and $\Psi_{0}:=\Psi_{\eps_{0}}$ as we set in \eqref{psieps} and \eqref{beps}.
With notations in Subsection \ref{subsec2.3}, let 
\begin{align*}
&J:= \{(\mb{x},\mb{y})\in \bZ^{m}\times\bZ^{n} : \|\mb{y}\|\leq m\|\mb{x}\| \text{ and }|\mb{b}\cdot\mb{x}|_{\bZ}>\eps_{0} \},\quad \Psi(h):=\frac{\Psi_{0}(h)}{h},\\ 
&\al:=(\mb{x},\mb{y})\in J,\quad \beta_{\al}:= \|\mb{x}\|,\quad R_{\al}:=\{A\in [0,1]^{mn} : {^{t}A}\mb{x}=\mb{y}\}.
\end{align*}
Note that $W_{\mb{b},\eps_{0}}=\Lam(\Psi)$ and the family $\cR$ of resonant sets $R_{\al}$ consists of $(m-1)n$-dimensional, rational hyperplanes.

By Lemma \ref{dualsum}, we may assume that $\sum_{h=1}^{\infty} h^{m+n-1} \left(\frac{\Psi_{0}(h)}{h}\right)^{s-n(m-1)}=\infty$.
Then we can find a strictly increasing sequence of positive integers $\{h_{i}\}_{i\in\bN}$ such that
\eq{
\sum_{h_{i-1}< h\leq h_{i}} h^{m+n-1} \left(\frac{\Psi_{0}(h)}{h}\right)^{s-(m-1)n} >1
} and $h_{i}>2h_{i-1}$. Put $\om(h):=i^{\frac{1}{n}}$ if $h_{i-1} < h\leq h_{i}$. Then $\om$ is $2$-regular and 
\eq{
\sum_{h=1}^{\infty} h^{m+n-1} \left(\frac{\Psi_{0}(h)}{h}\right)^{s-n(m-1)}\om(h)^{-n}=\infty.
}
For a constant $c>0$, define the ubiquitous function $\rho_{c}:\bR^{+}\to\bR^{+}$ by 
\eqlabel{UF}{
\rho_{c}(h):=
\begin{dcases}
    ch^{-\frac{1+n}{n}}       & \quad \text{if } m=1,\\
    ch^{-\frac{m+n}{n}}\om(h)   & \quad \text{if } m\geq 2,
\end{dcases}
}
Clearly the ubiquitous function is $2$-regular. 
\begin{thm}\label{LUb}
The pair $(\cR,\beta)$ is a locally ubiquitous system relative to $\rho=\rho_{c}$ for some constant $c>0$.
\end{thm}

\begin{proof}[Proof of Theorem \ref{thm2}]
For fixed $\mb{b}=(b_1,\dots.b_m)\in \bR^{m}\setminus\bZ^{m}$, assume that $b_i \notin \bZ$. If $b_i$ is rational, then there is $0<\eps<1/2$ such that $|kb_i|_{\bZ}>\eps$ for infinitely many positive integer $k$. If $b_i$ is irrational, then the set $\{kb_{i} \pmod{1} :k\in\bZ\}$ is dense in $[0,1]$. Hence, for any fixed $0<\eps<1/2$, $|kb_{i}|_{\bZ}>\eps$ holds for infinitely many positive integer $k$. Let us denote that increasing sequence by $(k_{j})_{j=1}^{\infty}$. This observation implies that the set $\{A\in[0,1]^{mn}: \|{^{t}Ak_{j}\mb{e}_i}\|_{\bZ}=0\}$, which is the finite union of $(m-1)n$-dimensional hyperplanes, is a subset of $W_{\mb{b},\eps}$ for each $j\in\bN$. Hence for any $0\leq s\leq (m-1)n$
\eq{
\cH^{s}(D_{m,n}^{\mb{b}}(\psi)^{c})\geq\cH^{s}(W_{\mb{b},\eps})=\cH^{s}([0,1]^{mn}).
}

Now assume that $(m-1)n<s\leq mn$. It follows from Theorem \ref{UbThm1} and Theorem \ref{LUb} that 
\eq{
\cH^{s}(D_{m,n}^{\mb{b}}(\psi)^{c})\geq\cH^{s}(W_{\mb{b},\eps_{0}})=\cH^{s}([0,1]^{mn}).
}
Here, we use the fact that the divergence and convergence of the sums
\eq{
\sum_{N=1}^{\infty} 2^{N\al}f(2^{N})\quad\text{and}\quad \sum_{h=1}^{\infty} h^{\al-1}f(h)\quad \text{coincide}
} for any monotonic function $f:\bR^{+}\to\bR^{+}$ and $\al\in\bR$.  
\end{proof}

Recall that we adopt the supremum norm $\|\cdot\|$ on $[0,1]^{mn}$ following Subsection \ref{subsec2.3}.
We consider $m=1$ and $m\geq 2$, separately.

\begin{proof}[Proof of Theorem \ref{LUb} for $m=1$]
Note that, for $(x,\mb{y})\in J$, the resonant set $R_{(x,\mb{y})}$ is the one point set $\{\frac{\mb{y}}{x}:=\left(\frac{y_1}{x},\dots,\frac{y_n}{x}\right)\}$ and $\Del(R_{x,\mb{y}}, \rho(2^{N}))=B(\frac{\mb{y}}{x}, \rho(2^N))$, the ball of radius $\rho(2^N)$ centered at $\frac{\mb{y}}{x}$. We basically follow the strategy in \cite[Chapter 3]{Tho04}.

Let $B$ an arbitrary square in $[0,1]^{n}$ and write $B=\prod_{i=1}^{n}[l_i , u_i]$, $\mb{l}=(l_1,\dots,l_n)$, $\mb{u}=(u_1,\dots,u_n)$. We restrict $\mb{y}$ to $\gcd(x,\mb{y})=1$ and $\frac{\mb{y}}{x}\in B$.
Observe that
\eqlabel{Ob1}{
|B\cap \Del(\rho,N)| \geq \left|\bigcup_{\substack{2^{N-1}< x\leq 2^{N}\\ |b\cdot x|_{\bZ}>\eps_{0}}}\bigcup_{\substack{x\mb{l}\leq\mb{y}\leq x\mb{u} \\\gcd(x,\mb{y})=1}}B\left(\frac{\mb{y}}{x},\rho(2^N)\right)\right|+O(\rho(2^{N})).
}
Here, $x\mb{l}< \mb{y}< x\mb{u}$ means that $xl_{i}< y_i < xu_{i}$ for all $1\leq i\leq n$.
Let 
\begin{align*}
&T(N):=\left\{\frac{\mb{y}}{x}\in\bQ^{n}: (x,\mb{y})\in J,\ \gcd(x,\mb{y})=1,\ x\mb{l}\leq\mb{y}\leq x\mb{u},\ 2^{N-1}< x\leq 2^{N} \right\}, \\
&G(N):=\left\{\frac{\mb{y}}{x}\in T(N) : B\left(\frac{\mb{y}}{x}, \rho(2^N)\right)\cap B\left(\frac{\mb{s}}{t}, \rho(2^N)\right)=\varnothing,\ \forall \frac{\mb{s}}{t} \left(\neq \frac{\mb{y}}{x}\right)\in T(N)\right\}.
\end{align*}

\begin{lem}\label{CountLem}
For $N$ large enough
\begin{enumerate}
\item\label{CL1} $\# T(N) \geq c_{1}|B|2^{N(n+1)}$ for some constant $0<c_{1}<1$.
\item\label{CL2} $\# G(N) \geq \frac{1}{2} \# T(N)$.
\end{enumerate}
\end{lem}

Thus, it follows from Lemma \ref{CountLem} that for $N$ large enough
\begin{align*}
\text{r.h.s. of } \eqref{Ob1} &\geq \left|\bigsqcup_{\frac{\mb{y}}{x}\in G(N)} B\left(\frac{\mb{y}}{x},\rho(2^N)\right)  \right|+O(\rho(2^{N}))\\
&=\# G(N)\times 2^n \rho(2^N)^{n}+O(\rho(2^{N})) \\
&\geq \frac{1}{2} \# T(N)\times 2^{n}\rho(2^N)^{n}+O(\rho(2^{N})) \\
&\geq c^{n}c_{1}2^{n-1} |B|+O(\rho(2^{N})) \geq \frac{1}{2} c^{n}c_{1}2^{n-1} |B|.
\end{align*}
Thus the local ubiquity follows from \eqref{Ob1}.
\begin{proof}[Proof of (\ref{CL1}) in Lemma \ref{CountLem}]
Note that for $\al >0$ and $\ell \in \bN$
\eqlabel{C1eq}{
\begin{split}
&\sum_{\substack{1\leq k\leq \al\ell \\ \gcd(k,\ell)=1}} 1 = \sum_{1\leq k \leq \al\ell}\sum_{d|\gcd(k,\ell)}\mu(d) =\sum_{d|\ell}\mu(d)\sum_{1 \leq k' \leq \al\ell/d} 1\\
&= \sum_{d|\ell} \mu(d)\lfloor\al\ell/d\rfloor = \al \vphi(\ell) + O(\tau(\ell)). 
\end{split}
}
where $\tau(\ell)=\sum_{d|\ell}1$, the number of divisors of $\ell$.
Here and hereafter, $\mu$, $\vphi$, and $\lfloor\cdot\rfloor$ stand for the M\"{o}bius function, Euler function, and floor function, respectively.

Fix small $0<\eps<\frac{3}{\pi^{2}}-\frac{1}{4}$. Note that $\frac{1}{N^2}\sum_{q=1}^{N} \vphi(q) \to \frac{3}{\pi^2}$ as $N\to\infty$ (see \cite[Theorem 330]{HW60}) and $\tau(h)=O(h^\del)$ for any $\del>0$ (see \cite[Theorem 315]{HW60}). Thus, for $N$ large enough and for $\del>0$ small enough,
\begin{align*}
\# T(N) &= \sum_{\substack{2^{N-1}< x\leq 2^N \\ |b\cdot x|_{\bZ}>\eps_{0}}}\sum_{\substack{x\mb{l}\leq\mb{y}\leq x\mb{u} \\ \gcd(x,\mb{y})=1}} 1 \geq \sum_{\substack{2^{N-1}< x\leq 2^N \\ |b\cdot x|_{\bZ}>\eps_{0}}}\sum_{\substack{xl_i \leq y_i \leq xu_i \\ i=2,\dots,n}}\sum_{\substack{xl_1 \leq y_1 \leq xu_1 \\ \gcd(x,y_1)=1}} 1\\
&\geq \sum_{\substack{2^{N-1}< x\leq 2^N \\ |b\cdot x|_{\bZ}>\eps_{0}}}\left( |B|\vphi(x)x^{n-1} +O(x^{n-1}\tau(x))\right) \\
&\geq \sum_{\substack{2^{N-1}< x\leq 2^N \\ |b\cdot x|_{\bZ}>\eps_{0}}}|B|\vphi(x)2^{(N-1)(n-1)} + O(2^{N(n+\del)})\\
&\geq |B|2^{(N-1)(n-1)}\left(\sum_{2^{N-1}< x\leq 2^N}\vphi(x) - \sum_{\substack{2^{N-1}< x\leq 2^N \\ |b\cdot x|_{\bZ}\leq \eps_{0}}}x \right) + O(2^{N(n+\del)})\\
&\geq |B|2^{(N-1)(n-1)}\left(\frac{3}{\pi^2}-\frac{1}{4}-\eps\right)(2^{2N}-2^{2(N-1)}) =c_{1}|B|2^{N(n+1)}.
\end{align*}
The second line is by \eqref{C1eq} and the fifth line is by Lemma \ref{Count}.
\end{proof}
\begin{proof}[Proof of (\ref{CL2}) in Lemma \ref{CountLem}]
Let $B(N):=T(N)\setminus G(N)$. By definition, $\frac{\mb{y}}{x}\in B(N)$ if and only if there is a point $\frac{\mb{s}}{t}\left(\neq\frac{\mb{y}}{x}\right)\in T(N)$ such that 
\eq{
B\left(\frac{\mb{y}}{x}, \rho(2^N)\right)\cap B\left(\frac{\mb{s}}{t},\rho(2^N)\right) \neq \varnothing.
} The coprimeness condition ensures that the centers $\frac{\mb{y}}{x}$ and $\frac{\mb{s}}{t}$ of the balls are distinct. Thus, we have $0 < \left\|\frac{\mb{y}}{x}-\frac{\mb{s}}{t} \right\| \leq 2\rho(2^N)$, or, equivalently,
\eq{
0<\|t\mb{y}-x\mb{s}\| \leq 2xt\rho(2^N).
}
It follows that the associated $4$-tuple $(\mb{y},x,\mb{s},t)$ is an element of the set
\begin{align*}
V(N):=\{(\mb{y},x,\mb{s},t):\ & 0<\|t\mb{y}-x\mb{s}\| \leq 2^{2N+1}\rho(2^N),\ \gcd(x,\mb{y})=\gcd(t,\mb{s})=1,\\
& 2^{N-1}< x,t \leq 2^N,\ x\mb{l}\leq \mb{y} \leq x\mb{u},\ t\mb{l}\leq \mb{s} \leq t\mb{u}
\}
\end{align*}
Hence, $\# B(N)\leq \# V(N)$ and it is enough to show that $\# V(N)< \frac{1}{2}\# T(N)$. 
Observe that if $n=1$, then $V(N)$ is empty by taking $c<\frac{1}{2}$. 
We consider $n=2$ and $n>2$, separately.

\medskip\noindent \textbf{Case $\mathbf{n=2}$.} 
Note that $2^{2N+1}\rho(2^N)=2c2^{N/2}$. If $(\mb{y},x,\mb{s},t)\in V(N)$, then there exist $a_1, a_2$ with $|a_i|\leq 2c2^{N/2}$ and at least one of $a_i$'s being nonzero, such that $ty_i - xs_i = a_i$ for all $i=1,2$.
Let us denote by $V(a_1,a_2,N)$ the set of the above $(\mb{y},x,\mb{s},t)\in V(N)$ for given $a_1,a_2$.

We first consider the case either $a_1=0$ or $a_2=0$. Given $2^{N-1}<x,t\leq 2^N$ and $a$, the number of solutions $(y,s)\in[1,2^N]^2$ of the equation $ty-xs=a$ is less than $2\textrm{gcd}(x,t)$ since the general solution of this equation is of the form $(y_0+p\frac{x}{\textrm{gcd}(x,t)},s_0+p\frac{t}{\textrm{gcd}(x,t)})$ for $p\in\bZ$. It follows that the number of elements $(\mb{y},x,\mb{s},t) \in V(N)$ such that either $a_1=0$ or $a_2=0$ is bounded above by
\eqlabel{ayeq0}{
\begin{split}
&\sum_{1\leq |a_1| \leq 2c2^{\frac{N}{2}}}\# V(a_1,0,N) + \sum_{1\leq |a_2| \leq 2c2^{\frac{N}{2}}}\# V(0,a_2,N)\\
&\leq 4\displaystyle\sum_{1\leq a_1\leq 2c2^{\frac{N}{2}}}
\displaystyle\sum_{(x,t)\in(2^{N-1},2^N]^2}\#\set{(y_1,s_1): ty_1-xs_1=a_1}\#\set{(y_2,s_2): ty_2-xs_2=0}\\
&\leq 4\displaystyle\sum_{1\leq a_1\leq 2c2^{\frac{N}{2}}}
\displaystyle\sum_{\substack{(x,t)\in(2^{N-1},2^N]^2 \\ \textrm{gcd}(x,t)|a_1}}(2\textrm{gcd}(x,t))^2\\
&=16\displaystyle\sum_{1\leq d\leq 2^N}
\displaystyle\sum_{\substack{1\leq a_1\leq 2c2^{\frac{N}{2}}\\ d|a_1}}d^2\#\set{(x,t)\in(2^{N-1},2^N]^2: \textrm{gcd}(x,t)=d}\\
&\ll\displaystyle\sum_{1\leq d\leq 2^N}\frac{2^{\frac{N}{2}}}{d}d^2\left(\frac{2^{N-1}}{d}\right)^2=O(N2^{\frac{5}{2}N}).
\end{split}
}

We now consider the case $a_1\neq0$ and $a_2\neq 0$. 
Note that if $(\mb{y},x,\mb{s},t)\in V(a_1,a_2,N)$, then we have
\eqlabel{ayeq}{
a_1 y_2 - a_2 y_1 = kx
} for some $k\in\bZ$. 
Thus we will count the set of $(a_1,a_2,k,x,y_1,y_2)$ satisfying the equation \eqref{ayeq} where
$2^{N-1}<x\leq 2^N$, $l_i x \leq y_i \leq u_i x$, and $1\leq |a_i|\leq 2c2^{N/2}$ for $i=1,2$. 
Let us denote by $\bar{V}(N)$ the above set.
We will only present the counting for the case $a_1>0$ and $a_2>0$, but the counting estimates remains the same for the cases of the other signs, and the proof also still works similarly.

For fixed $a_1>0$ and $a_2>0$, let us count the set of $(k,x,y_1,y_2)$ such that $(a_1,a_2,k,x,y_1,y_2)\in \bar{V}(N)$. 
It follows from the equation \eqref{ayeq} and $l_i x \leq y_i \leq u_i x$ for $i=1,2$ that 
$$a_1 l_2 - a_2 u_1 \leq k \leq a_1 u_2 - a_2 l_1.$$
Denoting by $d=\gcd(a_1,a_2)$, it follows from the equation \eqref{ayeq} that $d|kx$. Thus we can write
$d=d_1 d_2$, where $d_1|k$ and $d_2|x$, and denote by $a'_i=a_i /d$ for $i=1,2$, $k'=k/d_1$, and $x'=x/d_2$. 
Then we have 
\eqlabel{ayeq2}{
a'_1 y_2 - a'_2 y_1 = k'x'.
} 
If $(\bar{y}_1, \bar{y}_2)$ is a solution of \eqref{ayeq2}, then the general solution of \eqref{ayeq2} is of the form
$(\bar{y}_1 + pa'_1 , \bar{y}_2+ pa'_2)$ with $p\in \bZ$. Hence the number of solution $(y_1,y_2)$ of \eqref{ayeq2}
with $l_i x \leq y_i \leq u_i x$ for $i=1,2$ is at most 
$$\min \left(\left\lceil\frac{(u_1 -l_1)x}{a'_1}\right\rceil,\left\lceil\frac{(u_2-l_2)x}{a'_2}\right\rceil\right) 
\leq 2 \min \left(\frac{(u_1 -l_1)x}{a'_1},\frac{(u_2-l_2)x}{a'_2}\right)$$
since $(u_i -l_i)x/a'_i \geq 1$ with $i=1,2$ for all large enough $N$.
Hence it follows that for any small enough $\delta>0$,
\[
\begin{split}
&\sum_{1\leq a_1, a_2 \leq 2c2^{N/2}} \# \{(k,x,y_1,y_2):(a_1,a_2,k,x,y_1,y_2)\in \bar{V}(N)\}\\
&\leq \sum_{1\leq a_1, a_2 \leq 2c2^{N/2}}
\sum_{d=d_1 d_2} 
\sum_{\substack{d_1|k, d_2|x\\ 2^{N-1} <x \leq 2^N \\a_1 l_2 - a_2 u_1 \leq k \leq a_1 u_2 - a_2 l_1}}
2 \min \left(\frac{(u_1 -l_1)x}{a'_1},\frac{(u_2-l_2)x}{a'_2}\right) \\
& \leq \sum_{1\leq a_1, a_2 \leq 2c2^{N/2}}\sum_{d=d_1 d_2}  
2 \left\lceil \frac{a_2(u_1-l_1)+a_1(u_2-l_2)}{d_1}\right\rceil \left\lceil \frac{2^N}{d_2} \right\rceil 
\min \left(\frac{(u_1 -l_1)x}{a'_1},\frac{(u_2-l_2)x}{a'_2}\right)\\
& \leq \sum_{1\leq a_1, a_2 \leq 2c2^{N/2}}\sum_{d=d_1 d_2}  
4 \left(\frac{a_2(u_1-l_1)+a_1(u_2-l_2)}{d_1}+1\right) \frac{2^N}{d_2}\min \left(\frac{(u_1 -l_1)2^N}{a'_1},\frac{(u_2-l_2)2^N}{a'_2}\right)\\
& \leq \sum_{1\leq a_1, a_2 \leq 2c2^{N/2}}\sum_{d=d_1 d_2} 
2^{2N+2}\left((a_2(u_1-l_1)+a_1(u_2-l_2)) \min \left(\frac{u_1 -l_1}{a_1},\frac{u_2-l_2}{a_2}\right)+\frac{(u_1 -l_1)d_1}{a_1}\right) \\ 
& \leq \sum_{1\leq a_1, a_2 \leq 2c2^{N/2}}\tau(d)\left(2^{2N+3}(u_1-l_1)(u_2-l_2)
+ d 2^{2N+2}\frac{(u_1 -l_1)}{a_1}\right)\\
& \ll \sum_{1\leq d\leq 2c2^{N/2}}\sum_{1\leq a'_1, a'_2 \leq \frac{2c2^{N/2}}{d}} 
d^{\delta}\left(2^{2N+3}|B| + 2^{2N+2}\frac{(u_1 -l_1)}{a'_1}\right)\\
& \ll \sum_{1\leq d\leq 2c2^{N/2}} \left(\frac{1}{d^{2-\delta}}c^2 2^{3N}|B|
 + \frac{1}{d^{1-\delta}}N2^{2N+\frac{N}{2}}\right)\\
& \ll c^2 |B| 2^{3N} + O(N2^{2N+\frac{(1+\delta)N}{2}}).
\end{split}
\]
Combining with the cases of other signs, we have
\eqlabel{ayeq3}{
\#\bar{V}(N)\ll c^2|B|2^{3N}+O(N2^{2N+\frac{(1+\delta)N}{2}}).
}

We next claim that $\sum_{1\leq |a_1|,|a_2|\leq 2c2^{N/2}}\#  V(a_1,a_2,N)\leq 2\# \bar{V}(N)$ 
by showing that for fixed $(a_1,a_2,k,x,y_1,y_2)\in\bar{V}(N)$, there are at most two pairs of $(s_1,s_2,t)$ such that $(\by,x,\bs,t)\in V(N)$. To see this, observe that $ty_i \equiv a_i\ (\mathrm{mod}\ x)$ for $i=1,2$ and $\gcd(\by,x)=1$. Since $\gcd(\by,x)=1$, there exist $\alpha_1,\alpha_2\in\bZ$ such that $\alpha_1 y_1+\alpha_2 y_2\equiv 1 (\mathrm{mod}\ x)$. It follows that $t\equiv a_1\alpha_1+a_2\alpha_2$ is uniquely determined modulo $x$ for fixed $a_i,y_i$ and $x$. Since $t<2x$, the number of possible $t$ is at most two. Once $t$ is determined, then $s_1$ and $s_2$ are also determined uniquely, thus the claim follows.

Hence, combining \eqref{ayeq0}, \eqref{ayeq3}, and the above claim, we have 
$$\# V(N) \ll c^2|B|2^{3N}+O(N2^{2N+\frac{(1+\delta)N}{2}}+N2^{\frac{5}{2}N}).$$
By taking $\delta<1$, for all large enough $N$, $ \# V(N) \leq Cc^2 |B|2^{3N}$ for some absolute constant $C>0$.  
It follows that $\# V(N)<\frac{1}{2}\# T(N)$ for sufficiently large $N$ by choosing
$c<(\frac{c_1}{2C})^{1/2}$.

\medskip\noindent \textbf{Case $\mathbf{n>2}$.} For fixed $2^{N-1}<x,t\leq 2^N$, we denote by 
$$ d=\gcd(x,t),\ x'=\frac{x}{d},\ t'=\frac{t}{d},\ A=2^{2N+1}\rho(2^N),\text{ and } A'=\frac{A}{d}.$$

We will count the following set: for $0\leq a\leq A'$,
$$V_{x,t}(a):=\{(y,s): t'y-x's = a,\ x\ell \leq y \leq xu,\ t\ell \leq s \leq tu\}.$$
\textbf{Claim 1}. $\# V_{x,t}(0) \leq \max(\lceil d(u-\ell) \rceil, 1)$.
\begin{proof}
Since $(t',x')=1$, $x'|y$ holds. Thus
$$\# \{y: x'|y,\ x\ell \leq y\leq xu\} \leq \max\left(\left\lceil\frac{x(u-\ell)}{x'}\right\rceil,1\right)=\max(\lceil d(u-\ell) \rceil, 1),$$
which concludes the claim since $s$ is uniquely determined by $y$.
\end{proof}
\noindent

Now assume that $a\neq 0$ and $A' \geq 1$.
Let $y_0=y_0 (x,t)$ and $s_0=s_0 (x,t)$ be the integers with the smallest absolute value such that 
\[
t'y_0 \equiv a_0\ (\mathrm{mod}\ x')\quad \text{and}\quad  x's_0 \equiv -b_0\ (\mathrm{mod}\ t'),
\] for some $0<a_0=a_0 (x,t)\leq A'$ and $0<b_0=b_0 (x,t)\leq A'$. We remark that such $y_0$ and $s_0$ are unique since $a_0\neq 0$.\vspace{0.3cm}\\
\textbf{Claim 2}. $a_0 = b_0$ and $t'y_0 - x's_0 = a_0$.
\begin{proof}
Let $y,s$ be such that $t'y_0 -x's=a_0$ and $t'y-x's_0=b_0$. Then
\[
\begin{split}
|s|=\frac{|t'y_0-a_0|}{x'}\leq \frac{a_0 + t'|y_0|}{x'}\leq \frac{a_0+t'|y|}{x'}&=\frac{a_0 +|b_0 +x's_0|}{x'}\\
&\leq \frac{a_0+b_0}{x'}+|s_0|.
\end{split}
\]
Since $n> 2$, for all large enough $N$,
$$\frac{a_0+b_0}{x'}\leq \frac{2A'}{x'}=\frac{2A}{x}\leq \frac{2c2^{(1-\frac{1}{n})N}}{2^{N-1}}<1.$$ 
Hence we have $|s|=|s_0|$, and similarly $|y|=|y_0|$. If $s=-s_0$, then on one hand, $x's\equiv b_0 (\mathrm{mod}\ t')$; on the other hand, since $t'y_0 -x's=a_0$, we have $x's\equiv -a_0\equiv x'-a_0 (\mathrm{mod}\ t')$. It cannot happen that $b_0=x'-a_0$ since $x'>2A'\ge a_0+b_0$. Hence we get $s=s_0$, and similarly $y=y_0$. It concludes the claim.
\end{proof}
\noindent
\textbf{Claim 3}. $1\leq |y_0|, |s_0| \leq \left\lceil \frac{2x'}{A'} \right\rceil = \left\lceil \frac{2x}{A} \right\rceil \leq 2^{\frac{N}{n}+1}$.
\begin{proof}
Consider the set $P=\left\{t',2t',\dots,\left\lceil \frac{2x'}{A'} \right\rceil t'\right\}$ modulo $x'$. Partition $[1,x']\cap\bN$ into $\lfloor A' \rfloor$ consecutive integers. Then the number of the partitions is at most $\left\lceil \frac{x'}{\lfloor A' \rfloor} \right\rceil$. 
It follows from $A'\geq 1$ that $2\lfloor A' \rfloor \geq A'$, hence,
$\left\lceil \frac{x'}{\lfloor A' \rfloor} \right\rceil \leq \left\lceil \frac{2x'}{A'} \right\rceil$.
By the pigeonhole principle, there are at least two elements of $P$ in the same partition, say $it'$ and $jt'$ with $i\neq j$. Then
$(i-j)t' \ (\mathrm{mod}\ x')$ is contained in $[1,\lfloor A' \rfloor]$ or $[-\lfloor A' \rfloor,-1]$. The fact that $|i-j|\leq \left\lceil \frac{2x'}{A'} \right\rceil$ and the minimality of $y_0$ imply the claim for $y_0$. Similarly, we can conclude the claim for $s_0$.
\end{proof}
\noindent
\textbf{Claim 4}. If $|y_0|\leq (u-\ell)a_0$ or $|s_0|\leq (u-\ell)a_0$, then 
\[
\sum_{a=1}^{A'}\#V_{x,t}(a) \leq 10A(u-\ell).
\]
\begin{proof}
It suffices to show the case $|y_0|\leq (u-l)a_0$. Let $$\cA_k:=\set{xl\leq y\leq xu: y\equiv k(\mathrm{mod}\ y_0)}=\set{z_k,z_k+|y_0|,\cdots,z_k+\alpha_k |y_0|}.$$
for $0\leq k\leq |y_0|-1$. Then $z_k\in\bN$ is the element such that $xl\leq z_k<xl+|y_0|$ and $z_k\equiv k(\mathrm{mod}\ y_0)$, and $\alpha_k\in\bN$ satisfies $\alpha_k\leq\frac{x(u-l)}{|y_0|}$. Partition $\cA_k$ into $M=\lfloor\frac{x'}{a_0}\rfloor$ consecutive integers. Recall that $\frac{x'}{a_0}\ge\frac{x'}{A'}\ge 1$ holds. Then the number of the partitions is at most $\lceil\frac{\alpha_k+1}{M}\rceil\leq\frac{\frac{x(u-l)}{|y_0|}+1}{\frac{x'}{2a_0}}+1$.

Let $P=\set{z_k',z_k'+|y_0|,\cdots,z_k'+(M-1)|y_0|}$ be a partition of $M$ consecutive integers, where $z_k'\in\cA_k$. To count the number of $y\in P$ such that $t'y\equiv a$ $(\mathrm{mod}\ x')$ with $1\leq a\leq A'$, we see the set $t'P$ modulo $x'$.  Write $t'z_k'\equiv w (\mathrm{mod}\ x')$ for some $0\leq w< x'$. Then the elements of $t'P$ can be written $\set{w,w+a_0,\cdots, w+(M-1)a_0}$ or $\set{w-a_0,\cdots, w-(M-1)a_0}$ $(\mathrm{mod}\ x')$ depending on the sign of $y_0$. Since $Ma_0=\lfloor\frac{x'}{a_0}\rfloor a_0\leq x'$, there are at most $\lceil\frac{A'}{a_0}\rceil$ elements in $t'P$ which are congruent to $a$ modulo $x'$ for some $1\leq a\leq A'$.

To sum up, for each $\cA_k$, there are at most
\eq{\begin{aligned}
\lceil\frac{A'}{a_0}\rceil\cdot\lceil\frac{\alpha_k+1}{M}\rceil
&\leq\frac{2A'}{a_0}\left(\frac{2a_0x(u-l)+2a_0|y_0|}{x'|y_0|}+1\right)\\
&=\frac{4A'd(u-l)}{|y_0|}+\frac{4A'}{x'}+\frac{2A'}{a_0}
\end{aligned}}
number of $y$ such that $y\in\cA_k$ and $t'y\equiv a (\mathrm{mod}\ x')$ for some $1\leq a\leq A'$. Since there are $|y_0|$ number of $\cA_k$'s and $s$ is uniquely determined by $y$, we have
\eq{\begin{aligned}
\displaystyle\sum_{a=1}^{A'}\#V_{x,t}(a)&\leq |y_0|\left(\frac{4A'd(u-l)}{|y_0|}+\frac{4A'}{x'}+\frac{2A'}{a_0}\right)\\
&=4A(u-l)+\frac{4A|y_0|}{x}+\frac{2A'|y_0|}{a_0}\\
&\leq 4A(u-l)+4A(u-l)\frac{a_0}{x}+2A'(u-l)\leq 10A(u-l).
\end{aligned}}
Here we used the assumption $|y_0|\leq(u-l)a_0$ in the last line.
\end{proof}
We remark that under the assumption of \textbf{Claim 4}, the counting of $V_{x,t}(a)$'s is good enough for our purpose. Thus we will count the set of $x,t$'s such that $y_0,s_0,a_0$ may not satisfy the assumption of \textbf{Claim 4}.

Note that $\gcd(y_0,s_0)=1$, otherwise it contradicts to the minimality of $y_0,s_0$. Through \textbf{Claim 3}, we consider the following sets and the map: 
\[
\begin{split}
S_{\text{good}}=\{(y_0,s_0,a_0):\ & |y_0|\leq 2^{\frac{N}{n}+1},\ |s_0|\leq 2^{\frac{N}{n}+1},\ \gcd(y_0,s_0)=1,\\
& (u-\ell)^{-1}\min(|y_0|,|s_0|)\leq a_0 \leq A' 
\},\\
S_{\text{bad}}=\{(y_0,s_0,a_0):\ & |y_0|\leq 2^{\frac{N}{n}+1},\ |s_0|\leq 2^{\frac{N}{n}+1},\ \gcd(y_0,s_0)=1,\\
& 1\leq a_0 < (u-\ell)^{-1}\min(|y_0|,|s_0|) 
\},\\
\pi : (2^{N-1},2^N]^2 \ni (x,t) &\mapsto (y_0(x,t),s_0(x,t),a_0(x,t)) \in S_{\text{good}}\cup S_{\text{bad}}.
\end{split}
\]

Let us first count the set $\pi^{-1}(S_{\text{bad}})$.
For $(y_0,s_0,a_0)\in S_{\text{bad}}$, assume that there exists $t_0', x_0'$ such that $t_0'y_0-x_0's_0=a_0$. Since $\gcd(y_0,s_0)=1$, all solutions of $t'y_0-x's_0=a_0$ can be represented in the form 
\[
(t',x')=(t_0'+ks_0,x_0'+ky_0), \quad k\in\bZ.
\]
Thus, for each $d\geq 1$,
\[
\#\left\{(x,t)\in (2^{N-1},2^N]^2  : y_0(x,t)=y_0,\ s_0(x,t)=s_0,\ \gcd(x,t)=d\right\}\leq \min\left(\frac{2^N}{d|s_0|},\frac{2^N}{d|y_0|}\right).
\]
Summing over $1\leq d \leq 2^N$, we have
\[
\begin{split}
\#&\left\{(x,t)\in (2^{N-1},2^N]^2  : y_0(x,t)=y_0,\ s_0(x,t)=s_0\right\}\\
&\leq \sum_{1\leq d\leq 2^N}\frac{1}{d}\min\left(\frac{2^N}{|s_0|},\frac{2^N}{|y_0|}\right)
\ll N\min\left(\frac{2^N}{|s_0|},\frac{2^N}{|y_0|}\right).
\end{split}
\]
Since $n> 2$, it follows that for all small enough $\de >0$, 
\eqlabel{countbad}{
\begin{split}
\#\pi^{-1}(S_{\text{bad}})
%\# \left\{ (x,t)\in (2^{N-1},2^N]^2 : (y_0(x,t),s_0(x,t),a_0(x,t))\in S_{\text{bad}}\right\}
&\ll \sum_{\substack{|y_0|,|s_0|\leq 2^{\frac{N}{n}+1} \\ 1\leq a_0<(u-\ell)^{-1}\min(|y_0|,|s_0|)}} N\min\left(\frac{2^N}{|s_0|},\frac{2^N}{|y_0|}\right)\\
&\leq \sum_{|y_0|,|s_0|\leq 2^{\frac{N}{n}+1}} N(u-\ell)^{-1}\min\left(\frac{2^N}{|s_0|},\frac{2^N}{|y_0|}\right) \min(|y_0|,|s_0|)\\
&\leq (u-\ell)^{-1}\sum_{|y_0|,|s_0|\leq 2^{\frac{N}{n}+1}} N 2^N \min\left(\frac{|y_0|}{|s_0|},\frac{|s_0|}{|y_0|}\right)\\
&\ll (u-\ell)^{-1}\frac{N^2}{n}2^{N+\frac{2N}{n}} \ll 2^{(2-\de)N}.
\end{split}
}

Now, for each $1\leq i\leq n$ and $0\leq a\leq A'$, let us denote by 
\[
V_{x,t}^i (a):=\{(y_i,s_i): t'y_i-x's_i = a,\ x\ell_i \leq y_i \leq xu_i,\ t\ell_i \leq s_i \leq tu_i\}.
\]
%Using above observations in each $i$, for $1\leq j_1<\cdots<j_k\leq n$, we have
%\begin{align*}
%\sum_{(x,t)\in (2^{N-1},2^N]^2} \prod_{i=j_1,\dots,j_k} \sum_{a=1}^{A}V_{x,t}^i (a)
%&\ll \left(\# \pi^{-1}(S^{c})\right)\prod_{i=j_1,\dots,j_k} (u_{i}-l_{i})A  + \left(\# \pi^{-1}(S)\right)A^n\\
%&\ll 2^{2N}A^k \prod_{i=j_1,\dots,j_k} (u_{i}-l_{i}). 
%\end{align*}
Then we have
$$\# V(N) \leq \displaystyle\sum_{(x,t)\in (2^{N-1},2^N]^2} \prod_{i=1}^n \sum_{a=0}^{A'} \#V_{x,t}^i (a).$$

For $(x,t)\in\pi^{-1}(S_{\text{good}})$ and sufficiently large $N$, 
\eqlabel{countgood}{
\sum_{a=1}^{A'} \#V_{x,t}^i (a) + \#V_{x,t}^i (0)\leq 10A(u_i-l_i)+\max(A(u_i-l_i),1)\leq 11A(u_i-l_i).
} 
We applied \textbf{Claim 4} for the first term, and \textbf{Claim 1} for the second term. 

For each $0\leq a\leq A'$, the number of solutions $(y,s)$ of
$$t'y-x's = a,\ 1 \leq y \leq x,\ 1 \leq s \leq t $$
is at most $d$, hence $\#V_{x,t}^i (a)\leq d$. For $(x,t)\in\pi^{-1}(S_{\text{bad}})$, it follows that 
\eqlabel{countbad2}{
\sum_{a=0}^{A'} \#V_{x,t}^i (a)\leq (A'+1)d\leq 2A.
}
Therefore, combining \eqref{countbad}, \eqref{countgood}, and \eqref{countbad2}, we have
\[
\begin{split}
\# V(N)&\ll \left(\# \pi^{-1}(S_{\text{good}})\right) \prod_{i=1}^n A(u_i-\ell_i ) + \left(\# \pi^{-1}(S_{\text{bad}})\right) (2A)^n\\
&\ll 2^{2N}A^n \prod_{i=1}^{n} (u_i -\ell_i) + O(2^{N(n+1-\delta)})\ll c^{n}|B|2^{N(n+1)},
\end{split}
\]
hence $\# V(N)\leq Cc^n|B|2^{N(n+1)}$ for sufficiently large $N$ and some absolute constant $C>0$. It follows that $\# V(N)<\frac{1}{2}\# T(N)$ for sufficiently large $N$ by choosing $c<(\frac{c_1}{2C})^{1/n}$.
\end{proof}
This proves Theorem \ref{LUb} for $m=1$.
\end{proof}

\begin{proof}[Proof of Theorem \ref{LUb} for $m\geq 2$]
Note that it suffices to show that
\eqlabel{WTS}{
|\Del(\rho,N)|\to 1 \quad\text{as}\quad N\to\infty
} for the local unbiquity. Instead of the strategy for $m=1$, we will use mean and variance techniques in \cite{DV97} using the auxiliary function $\om$ in \eqref{UF}.

Without loss of generality we may assume that $\eps_{0}=\eps(\mb{b})=\frac{|b_{1}|_{\bZ}}{4}$.    
Let  $I(N)$ denote the set of vectors $\mb{x}=(x_{1},\dots,x_{m})\in \bZ^{m}$ such that
\begin{enumerate}
\item\label{As1} $N \leq x_1 \leq 2N$,
\item\label{As2} For $i=2,\dots,m$,
\eq{
1\leq x_{i}\leq \frac{N}{\om(2N)^{\frac{1}{2(m-1)}}},
}
\item\label{As3} $\gcd(\mb{x})=1$,
\item\label{As4} $|\mb{b}\cdot\mb{x}|_{\bZ}>\eps_{0}$.
\end{enumerate}
Denote by $J(N):=\{(\mb{x},\mb{y})\in I(N)\times \bZ^{n}: \|\mb{y}\|\leq m\|\mb{x}\|\}$. Then $J(N)\subset J$.

Let $\chi_{\Del_{(\mb{x},\mb{y})}}$ be the characteristic function 
\eq{
\chi_{\Del_{(\mb{x},\mb{y})}}(A):=
\begin{dcases}
    1 & \quad \text{if } A\in \Del_{(\mb{x},\mb{y})},\\
    0 & \quad \text{otherwise},
\end{dcases}
}
where 
\eq{
\Del_{(\mb{x},\mb{y})} := \Del(R_{(\mb{x},\mb{y})}, N^{-\frac{m}{n}}\om(2N)\|\mb{x}\|^{-1}).
}
Also, for a matrix $A \in [0,1]^{mn}$, define
\eq{
\nu_{N}(A):= \sum_{(\mb{x},\mb{y})\in J(N)} \chi_{\Del_{(\mb{x},\mb{y})}}(^{t}A).
}
Thus $\nu_{N}(A)$ is the number of resonant sets $R_{(\mb{x},\mb{y})}$ for $(\mb{x},\mb{y})\in J(N)$ which are `close' to $^{t}A$, i.e. such that $\|{^{t}A}\mb{x}-\mb{y}\|<\de(N)$, where $\de(N):= N^{-\frac{m}{n}}\om(2N)$. Denote by $\mu_{N}$ and $\sigma_{N}^{2}$ the mean and variance respectively, that is
\eq{
\mu_{N}:= \int_{[0,1]^{mn}}\nu_{N}(A)dA \quad\text{and}\quad \sigma_{N}^{2}:= \int_{[0,1]^{mn}} \nu_{N}^{2}(A)dA - \mu_{N}^{2}.
}

Since $\|\mb{x}\|^{-1}\leq N^{-1}$ for any $\mb{x}\in I(N)$, we have 
\eq{
\Del(R_{(\mb{x},\mb{y})}, N^{-\frac{m}{n}}\om(2N)\|\mb{x}\|^{-1}) \subset \Del(R_{(\mb{x},\mb{y})}, \rho(2N)) 
}
by taking $c<2^{\frac{m+n}{n}}$. Thus, we claim that 
\eq{
|Z_{N}| \to 0 \quad\text{as}\quad N\to\infty,
}
where $Z_{N}:= \nu_{N}^{-1}(0)=\{A\in[0,1]^{mn}: \nu_{N}(A)=0\}$, which implies \eqref{WTS} by replacing $N$ with $2^{N-1}$.

\begin{lem}\label{MVlem}
For $N$ large enough, $\sigma_{N}^{2}\leq \mu_{N}$ and $\mu_{N}\geq c_{0}\om(2N)^{\frac{1}{2}}$ for some positive constant $c_{0}$ independent of $N$.
\end{lem}
\begin{proof}
Suppose that $N$ is large enough so that $\de(N)=N^{-\frac{m}{n}}\om(2N) < \frac{1}{2}$. 
By Lemma 8 in \cite{Spr79}, for $\mb{x} \in I(N)$,
\eq{
\begin{split}
\sum_{\mb{y}:(\mb{x},\mb{y})\in J(N)}\int_{[0,1]^{mn}}\chi_{\Del_{(\mb{x},\mb{y})}}(A)dA &= |\{A\in [0,1]^{mn}: \|{^{t}A}\mb{x}\|_{\bZ}< \de(N)\}|\\
=(2\de(N))^{n}= 2^{n}N^{-m}\om(2N)^{n}.
\end{split}
}
Hence 
\eq{
\mu_{N}=\sum_{(\mb{x},\mb{y})\in J(N)} \int_{[0,1]^{mn}}\chi_{\Del_{(\mb{x},\mb{y})}}(A)dA = \sum_{\mb{x}\in I(N)} 2^{n}N^{-m}\om(2N)^{n}.
}

Let $\cS(i)$ denote the set of vectors $\mb{x}\in\bZ^{m}$ satisfying the condition $(i)$ in the definition $I(N)$ for each $i=1,2,3,4$.
Note that
\eq{
\sum_{\mb{x}\in I(N)}1 \geq \sum_{\mb{x}\in \cS(\ref{As1}) \cap \cS(\ref{As2}) \cap \cS(\ref{As3}) }1 - \sum_{\mb{x}\in \cS(\ref{As1})\cap \cS(\ref{As2}) \cap \cS(\ref{As4})^{c} }1. 
}
Following \cite[p.40]{Spr79},
\eq{
\begin{split}
&\sum_{\mb{x}\in \cS(\ref{As1}) \cap \cS(\ref{As2}) \cap \cS(\ref{As3}) }1 = \sum_{\mb{x}\in \cS(\ref{As1}) \cap \cS(\ref{As2})}\sum_{d|\gcd(\mb{x})}\mu(d)\\
&=\sum_{N\leq x_{1}\leq 2N} \sum_{d|x_{1}}\mu(d) \prod_{i=2}^{m}\left|\left\{x_i \in \bZ : d|x_i,\ 1\leq x_{i} \leq \frac{N}{\om(2N)^{\frac{1}{2(m-1)}}}  \right\} \right| \\
&=\sum_{N\leq x_{1}\leq 2N} \sum_{d|x_{1}}\mu(d) \left\lfloor \frac{N}{d\om(2N)^{\frac{1}{2(m-1)}}}\right\rfloor^{m-1} \\
&=\sum_{N\leq x_{1}\leq 2N}\left( \frac{N^{m-1}}{\om(2N)^{\frac{1}{2}}}\sum_{d|x_{1}}\frac{\mu(d)}{d^{m-1}} + O\left(N^{m-2}\sum_{d|x_{1}}\frac{|\mu(d)|}{d^{m-2}}\right) \right)\\
&=\begin{dcases}
    \sum_{N\leq x_{1}\leq 2N}\frac{N}{\om(2N)^{\frac{1}{2}}}\frac{\vphi(x_{1})}{x_{1}}+O(\tau(x_{1}))       & \quad \text{if } m=2,\\
    \sum_{N\leq x_{1}\leq 2N}\frac{N^{m-1}}{\om(2N)^{\frac{1}{2}}}\prod_{\substack{p|x_1 \\ p \text{ prime}}}\left(1-\frac{1}{p^{m-1}}\right)+O(N^{m-2}\tau(x_{1}))   & \quad \text{if } m\geq 3.
\end{dcases}
\end{split}
}

Fix small $0<\eps<\frac{6}{\pi^{2}}-\frac{1}{2}$. Note that $\frac{1}{N}\sum_{q=1}^{N} \frac{\vphi(q)}{q}\to \frac{6}{\pi^2}$ as $N\to\infty$ (see \cite[Lemma 2.4]{H98}) and $\tau(h)=O(h^\del)$ for any $\del>0$ (see \cite[Theorem 315]{HW60}). In the case $m=2$, we have  
\eq{
\sum_{\mb{x}\in \cS(\ref{As1}) \cap \cS(\ref{As2}) \cap \cS(\ref{As3}) }1 \geq \left(\frac{6}{\pi^{2}}-\eps\right)\frac{N^{2}}{\om(2N)^{\frac{1}{2}}}
} for all large enough $N$.
If $m\geq 3$, then
\eq{
\prod_{\substack{p|x_1 \\ p \text{ prime}}}\left(1-\frac{1}{p^{m-1}}\right) > \prod_{p \text{ prime}}\left(1-\frac{1}{p^{2}}\right)=\frac{6}{\pi^2},
} hence we have that for all large enough $N$,
\eq{
\sum_{\mb{x}\in \cS(\ref{As1}) \cap \cS(\ref{As2}) \cap \cS(\ref{As3}) }1 \geq \left(\frac{6}{\pi^{2}}-\eps\right)\frac{N^{m}}{\om(2N)^{\frac{1}{2}}}.
}

On the other hand, it follows from Lemma \ref{Count} that
\eq{
\sum_{\mb{x}\in \cS(\ref{As1})\cap \cS(\ref{As2}) \cap \cS(\ref{As4})^{c} }1\leq \frac{1}{2}\frac{N^{m}}{\om(2N)^{\frac{1}{2}}}.
}
Taking $c_{0}= 2^{n}\left(\frac{6}{\pi^{2}}-\frac{1}{2}-\eps \right)>0$, it follows that 
\eq{
\mu_{N}=  2^{n}N^{-m}\om(2N)^{n}\sum_{\mb{x}\in I(N)} 1 \geq c_{0} \om(2N)^{\frac{1}{2}}.
}

To prove that $\sigma_{N}^{2}\leq \mu_{N}$, we note that, for $\mb{x}\neq\mb{x}' \in I(N)$, 
\eq{
\begin{split}
&\sum_{\mb{y}:(\mb{x},\mb{y})\in J(N)}\sum_{\mb{y}':(\mb{x}',\mb{y}')\in J(N)} \int_{[0,1]^{mn}} \chi_{\Del_{(\mb{x},\mb{y})}}(A)\chi_{\Del_{(\mb{x}',\mb{y}')}}(A)dA\\
&=  |\{A\in [0,1]^{mn}: \|{^{t}A}\mb{x}\|_{\bZ}< \de(N)\}|\times|\{A\in [0,1]^{mn}: \|{^{t}A}\mb{x}'\|_{\bZ}< \de(N)\}|\\
&= 2^{2n}N^{-2m}\om(2N)^{2n}.
\end{split}
} by Lemma 9 in \cite{Spr79}. Thus we have
\eq{
\begin{split}
& \int_{[0,1]^{mn}}\nu_{N}^{2}(A)dA \\
&= \sum_{\mb{x}\in I(N)}\sum_{\mb{x}'\in I(N)}\sum_{\mb{y}:(\mb{x},\mb{y})\in J(N)}\sum_{\mb{y}':(\mb{x}',\mb{y}')\in J(N)} \int_{[0,1]^{mn}} \chi_{\Del_{(\mb{x},\mb{y})}}(A)\chi_{\Del_{(\mb{x}',\mb{y}')}}(A)dA\\
&= \mu_{N} + 2^{2n}N^{-2m}\om(2N)^{2n} \sum_{\mb{x}\neq\mb{x}' \in I(N)} 1 \leq \mu_{N} + \mu_{N}^{2}.
\end{split}
}
By definition of $\sigma_{N}^{2}$, we have 
\eq{
\sigma_{N}^{2}\leq \mu_{N}.
}
\end{proof}

Note that
\eq{
\sigma_{N}^{2}=\int_{[0,1]^{mn}}\left(\nu_{N}(A)-\mu_{N}\right)^{2}dA \geq \int_{Z_{N}}\left(\nu_{N}(A)-\mu_{N}\right)^{2}dA = \mu_{N}^{2}|Z_{N}|.
}
This together with Lemma \ref{MVlem} implies that
\eq{
|Z_{N}| \leq \frac{1}{\mu_{N}} \to 0 \quad\text{as}\quad N\to\infty.
}
\end{proof}

\section{Concluding remarks and open questions}
Remarks and open questions about homogeneous cases are stated in \cite[Section 4]{KW18}, \cite[Section 7.1]{KW19}, and \cite[Section 1]{KSY21}, so we skip remarks about homogeneous cases by refering these two references.
\subsection{$A$-fixed singly metric case}
In this paper, we focused on the doubly metric case ($\wh{D}_{m,n}(\psi)$) and the $\mb{b}$-fixed singly metric case ($\wh{D}_{m,n}^{\mb{b}}(\psi)$). Thus, it is natural to ask the same question for $A$-fixed singly metric case. More precisely, for fixed $A\in M_{m,n}(\bR)$ let $$\wh{D}_{m,n,A}(\psi)=\{\mb{b}\in\bR^{m}: (A,\mb{b})\in \wh{D}_{m,n}(\psi)\}.$$ As stated in \cite[Section 7.3]{KW19}, $\wh{D}_{m,n,A}(\psi)$ seems to depend heavily on Diophantine properties of $A$. So, one can ask which Diophantine conditions on $A$ guarantee that $\wh{D}_{m,n,A}(\psi)$ has full or null Lebesgue (or Hausdorff) measure depending on $\psi$. It seems plausible to use the transference lemma (Lemma \ref{dual}).

\subsection{Dirichlet non-improvable set vs. well-approximable set}
It has been shown in \cite{HKWW18} that the set of $\psi$ Dirichlet non-improvable numbers is a bigger set than the set of $\psi$-approximable numbers. It is subsequently proved in \cite{BBH20-1} that the difference set is non-trivial, see also \cite{BBH20-2}. So, it is natural to ask whether it is reasonable to conclude that $W_{m,n}(\psi)$ is contained in the set $\wh{D}_{m,n}(\Psi)^c$ for an appropriate choice of $\Psi$. In particular, with the notations in Subsection \ref{subsec4.1}, it raises natural question of estimating the Hausdorff dimension of difference sets $\wh{D}_{m,n}(\psi)^c \setminus \limsup \wh{W}_{S,\eps}$ and $\wh{D}_{m,n}^{\mb{b}}(\psi)^c \setminus W_{\mb{b},\eps}$ as in \cite{BBH20-1, BBH20-2}. 

%\subsection{General norm and general dimension function}

\subsection{Weighted setting}
Let $\mathbf{r}=(r_1,\cdots,r_m)$ be an $m$-tuple and $\bs=(s_1,\cdots,s_n)$ be an $n$-tuple such that $r_i,s_j>0$ and $\displaystyle\sum_{1\leq i\leq m} r_i=1=\displaystyle\sum_{1\leq j\leq n} s_j$. Define the $\mathbf{r}$-quasinorm of $\bx\in\bR^m$ and $\bs$-quasinorm of $\by\in\bR^n$ by
$$\|\bx\|_{\mathbf{r}}:=\displaystyle\max_{1\leq i\leq m}|x_i|^{\frac{1}{r_i}} \mbox{  and }\|\by\|_\bs:=\displaystyle\max_{1\leq j\leq n}|y_j|^{\frac{1}{s_i}}.$$
For a non-increasing function $\psi:[T_0,\infty)\rightarrow \bR_{+}$, where $T_0>1$ is fixed, we say that $A\in M_{m,n}(\bR)$ is $\psi$-Dirichlet with respect to the weight $(\mathbf{r},\mathbf{s})$ if the system
$$\|A\bq-\bp\|_{\mathbf{r}}<\psi(T) \mbox{  and } \|\bq\|_\bs<T$$
has a nontrivial integral solution for all large enough $T$. 

In this paper we have computed the Hausdorff measure of Dirichlet non-improvable sets in the unweighted setting $\mathbf{r}=(\frac{1}{m},\cdots,\frac{1}{m}), \bs=(\frac{1}{n},\cdots,\frac{1}{n})$, so one can ask the same questions for arbitrary weights $(\mathbf{r},\mathbf{s})$. There have been several recent results on the weighted ubiquity and weighted transference theorems (see \cite{CGGMS20}, \cite{G20}, and \cite{WW21}). It seems plausible to utilize these results to obtain a criterion for the Hausdorff measure of the $\psi$-Dirichlet non-improvable set in the weighted setting, as we used the ubiquity and the mass transference principle in the unweighted setting.

\subsection{Diophantine exponents}
In Corollary \ref{exponent}, we computed the Hausdorff dimension of the set of pairs $(A,\mathbf{b})$ (or the set of $A$ for fixed $\mathbf{b}\in\bR^n\setminus\bZ^n$) with uniform Diophantine exponents $\hat{w}(A,\mathbf{b})\leq w$. It directly implies that the Hausdorff dimension of the level set of pairs $(A,\mathbf{b})$ with $\hat{w}(A,\mathbf{b})=w$ is $mn+m-\frac{n-mw}{1+w}$ as stated in Corollary \ref{exponent}, where $w\leq\frac{n}{m}$. However, the results of the present paper does not give the Hausdorff dimension of the level sets for $w>\frac{n}{m}$ since the $\psi$-Dirichlet non-improvable set with $\psi(q)=q^{-mw}$ has full Hausdorff dimension. Therefore, a study on the Hausdorff dimension of Dirichlet improvable sets (instead of non-improvable sets) would be necessary to compute the Hausdorff dimension of the level sets for $w>\frac{n}{m}$.

%\bibliographystyle{amsalpha}
%\bibliography{uribib}
\def\cprime{$'$} \def\cprime{$'$} \def\cprime{$'$}
\providecommand{\bysame}{\leavevmode\hbox to3em{\hrulefill}\thinspace}
\providecommand{\MR}{\relax\ifhmode\unskip\space\fi MR }
\providecommand{\MRhref}[2]{%
  \href{http://www.ams.org/mathscinet-getitem?mr=#1}{#2}
}
\providecommand{\href}[2]{#2}

\end{document}